%% file: inhomogeneous.tex
\documentclass[review,onefignum,onetabnum]{siamart171218}

\input{ex_shared}
\nolinenumbers

\ifpdf
\hypersetup{
  pdftitle={Asymptotic solutions of inhomogeneous differential equations having a turning point},
  pdfauthor={T. M. Dunster}
}
\fi

\begin{document}

\maketitle

\begin{abstract}
  Asymptotic solutions are derived for inhomogeneous differential equations having a large real or complex parameter and a simple turning point. They involve Scorer functions and three slowly varying analytic coefficient functions. The  asymptotic approximations are uniformly valid for unbounded complex values of the argument, and are applied to inhomogeneous Airy equations having polynomial and exponential forcing terms. Error bounds are available for all approximations, including new simple ones for the well-known asymptotic expansions of Scorer functions of large complex argument.
\end{abstract}

\begin{keywords}
  {Asymptotic expansions, Airy functions, Scorer functions, Turning point theory, WKB methods}
\end{keywords}

\begin{AMS}
  34E05, 33C10, 34E20
\end{AMS}

\section{Introduction} 
\label{sec1}

In this paper we obtain asymptotic expansions for particular solutions for differential equations of the form
\begin{equation}
\label{eq1}
d^{2}w/dz^{2}-\left\{u^{2}f(z)+g(z)\right\}w=p(z),
\end{equation}
where $u$ is a large parameter, real or complex, and $z$ lies in an unbounded complex domain $Z$, which is precisely defined in \cref{homogeneous}. In $Z$ the functions $f(z)$ and $g(z)$ are meromorphic, and $p(z)$ ("forcing term") is analytic, and for simplicity we assume all are independent of $u$ (an assumption that can readily be relaxed). We further assume that $f(z)$ has no zeros in $Z$ except for a simple zero at $z=z_{0}$, which is the turning point of the equation. Our expansions will be valid in certain subsets of $Z$.

Equations of this form have a number of physical applications, such as the study of flow through variable permeability media \cite{Alzahrani:2016:NKF}, and in toroidal shell problems \cite{Tumarkin:1959:ASO}. Mathematical applications include inhomogeneous forms of the confluent, Weber and Bessel differential equations, and our new approximations will be applied in subsequent papers to these equations.

The homogeneous case ($p(z)=0$) has been extensively studied in the literature (see \cite[Chap. 11]{Olver:1997:ASF}), and more recently in \cite{Dunster:2017:COA} and \cite{Dunster:2020:SEB} where new expansions were given, along with error bounds, that are readily computable. A summary of the main results of \cite{Dunster:2017:COA} and \cite{Dunster:2020:SEB} is given in \cref{homogeneous}, and we shall make use of these results for our study of the inhomogeneous case.

For previous results for the inhomogeneous equation see  \cite{Clark:1963:ASN} for the case where the independent variable lies in a finite interval. In \cite[Chap.11, Sect. 12]{Olver:1997:ASF} asymptotic expansions for solutions of (\ref{eq1}) with $u$ and $z$ complex are presented, but the coefficients in these expansions are hard to compute, and a full analysis including error bounds is not pursued. Subsequently, a rigorous study was given by Baldwin \cite{Baldwin:1991:CFF} who derived solutions which are superficially similar to ours, in particular involving three slowly-varying coefficient functions. Her method was based on that by Boyd \cite{Boyd:1987:AEF} who studied the related homogeneous equation, and involved solving certain differential equations for the coefficient functions. The asymptotic expansions and error analyses of both their methods are quite involved, and in general both are harder to compute than ours. We also mention that our conditions on $p(z)$ at infinity are less restrictive.

Our approach is quite different to \cite{Baldwin:1991:CFF}, and results in much simpler expansions and error bounds. A key difference is that in this paper we shall consider $z$ bounded away and close to $z_{0}$ separately. In the former case (\cref{away}) we use, and expand upon, the new theory of \cite{Dunster:2020:LGE}, specialized to the current case of a simple turning point. From this we define three fundamental particular solutions, which are uniquely defined as being slowly varying for large $u$ and $z$ in certain sectors of the $z$ plane. The feature of these asymptotic expansions is that they involve very simple coefficients, as well as explicit and readily computable error bounds. Moreover, we will exploit this in our extension to expansions valid at the turning point, where the Cauchy integral formula plays a prominent role.

In \cref{scorer} we consider Scorer functions, which are solutions of the inhomogeneous Airy equation having a constant forcing term (see (\ref{eq3}) below). These play an important role in our approximations that are valid near the turning point. They have well-known properties that can be found in the literature (see \cite[Sect. 9.12]{NIST:DLMF}), and for details on their computation see \cite{Gil:2000:ONI} and \cite{Gil:2002:F7R}. In \cref{scorer} definitions of numerically satisfactory forms in the complex plane are presented, and new error bounds are derived for their large argument asymptotic expansions which are simpler than those appearing in the literature. We mention that recently Nemes \cite{Nemes:2018:EBF} constructed error bounds for the same expansions, but they are not quite as simple since they were obtained as special case of the more general (and hence complicated) Lommel functions. His bounds can be constructed from \cite[Thm. 4, Eq. (52), p. 531, Appendix A]{Nemes:2018:EBF}.

In \cref{close} asymptotic solutions of (\ref{eq1}) are obtained that are valid in domains containing the turning point, and which involve the above-mentioned Scorer functions. As in \cite{Baldwin:1991:CFF} these asymptotic solutions contain three slowly varying analytic coefficient functions, but here they are very easy to compute. Two of these are precisely the ones appearing in the asymptotic solutions of the homogeneous version of (\ref{eq1}) originally given in \cite{Dunster:2017:COA}. The new third coefficient function is constructed for the particular solutions, and we show how it can be computed via Cauchy's integral formula in conjunction with our new asymptotic expansions given in \cref{away}. An integral part of our analysis are certain connection coefficients linking the three fundamental particular solutions and their homogeneous partners outlined in \cref{homogeneous}. In applications these coefficients are often known explicitly, but if not, we show how they can be asymptotically computed for large $u$, again using Cauchy's integral formula.

In \cref{sec6} we illustrate our new  asymptotic approximations with an application to inhomogeneous Airy equations having polynomial and exponential forcing terms. As a corollary we obtain uniform approximations for indefinite Laplace-type integrals involving the Airy function. We mention that in \cite{Miura:1985:PSO} the case of a polynomial forcing term for the Airy equation was also studied, but not asymptotically with the presence of a large parameter. 

\section{Solutions of homogeneous equation}
\label{homogeneous}

For homogeneous equation let us present the main results from \cite{Dunster:2017:COA} and \cite{Dunster:2020:SEB}. Firstly, we define $Z$ as a domain containing $z_{0}$ and in which $f(z)$ and $g(z)$ are meromorphic, and $p(z)$ is analytic. If $f(z)$ and $g(z)$ have poles at $z=w_{j} \in Z$ ($j=1,2,3,\cdots $) (say), then at $z=w_{j}$

(i) $f(z)$ has a pole of order $m>2$, and $g(z)$ is analytic or has a pole of order less than $\frac{1}{2}m+1$, or

(ii) $f(z)$ and $g(z)$ have a double pole, and 
$\left( z-w_{j}\right) ^{2}g(z) \rightarrow -\frac{1}{4}$ as $z\rightarrow w_{j}$.
We shall call these \textit{admissible poles}.

If $Z$ is unbounded we further assume that $f(z)$ and $g(z)$ can be expanded in convergent series in a neighborhood of $z=\infty$ in $Z$ of the form
\begin{equation}
f(z) ={z}^{m}\sum\limits_{s=0}^{\infty }f_{s}{z}^{-s},\ g\left(
z\right) ={z}^{p}\sum\limits_{s=0}^{\infty }g_{s}{z}^{-s},
\label{fginfinity}
\end{equation}
where $f_{0}\neq 0$, $g_{0}\neq 0$, and either $m$ and $p$ are integers such that $m>-2$ and $p<\frac{1}{2}m-1$, or $m=p=-2$ and $g_{0}=-\frac{1}{4}$. 

The above conditions on $f(z)$ and $g(z)$ at finite or unbounded singularities in $Z$ are sufficient for most applications to special functions. However these are not necessary, and can be relaxed somewhat; see \cite[Chap. 10, Sect. 5]{Olver:1997:ASF}.

The following variables play a central role in all our approximations, and come from the Liouville transformation on the homogeneous form of (\ref{eq1})
\begin{equation}
\xi =\tfrac{2}{3}\zeta^{3/2}= \int_{z_{0}}^{z}f^{1/2}(t) dt.
\label{xi}
\end{equation}
The branch in (\ref{xi}) will be specified in \cref{away}. Note that the turning point $z=z_{0}$ of (\ref{eq1}) is mapped to $\zeta =\xi=0$.

We denote by $\mathbf{Z}$ and $\Xi$ the $\zeta$ and $\xi$ domains, respectively, corresponding to $Z$, and furthermore define the sectors
\begin{equation}
\mathrm{\mathbf{S}}_{j}=\left\{\zeta :\left\vert {\arg \left( {\zeta e^{-2\pi ij/3}}\right) }\right\vert \leq \tfrac{1}{3}\pi
\right\} \ \left( {j=0,\pm 1}\right),
\label{eq6}
\end{equation}
and 
\begin{equation}
\mathrm{\mathbf{T}}_{j}=\left\{\zeta :\left\vert {\arg \left(u^{2/3} \zeta e^{-2\pi ij/3}\right) }\right\vert \leq \tfrac{1}{3}\pi
\right\} \ \left( {j=0,\pm 1}\right).
\label{eq6b}
\end{equation}
which is the sector $\mathbf{S}_{j}$ rotated by an angle $2\arg(u)/3$ in a negative direction. In the $z$ plane the sets corresponding to (\ref{eq6}) and (\ref{eq6b}) are denoted by $S_{j}$ and $T_{j}$ respectively.

For $j=0,\pm 1$ we let $\zeta^{(j)}$ be points at infinity in $\mathbf{T}_{j}\cap \mathbf{Z}$. The corresponding $\xi$ and $z$ points are denoted by $\xi^{(j)}$ and $z^{(j)}$, respectively. These will appear in our error bounds. The points $z^{(j)}$ are either at infinity or at an admissible pole in $Z$, and in either case $\xi^{(j)}$ is unbounded.

For each\ $j=0,\pm 1$\ we define subdomains of $Z$, denoted by $Z_{j}(u)$, which comprise the $z$ point set for which there is a path $\hat{\mathcal{L}}_{j}$ (say) linking $z$ with $z^{(j) }$ in $Z$ and having the properties

(i) $\hat{\mathcal{L}}_{j}$ consists of a finite chain of $R_{2}$ arcs (as defined in \cite[Chap. 5, sec. 3.3]{Olver:1997:ASF}), and

(ii) as $v$ passes along $\hat{\mathcal{L}}_{j}$ from $z^{(j) }$ to $z$, the real part of $u\xi(v)$ varies continuously and is monotonic, where $\xi(v)$ is given by (\ref{xi}) with $z=v$.

Following Olver \cite[Chap. 6, sec. 11]{Olver:1997:ASF} these are called \textit{progressive paths}. For (i) we will call this an $R_{2}$ \textit{path}. The choice of branch cut or cuts for $\xi$ are usually straight lines in the $z$ plane, but sometimes can be chosen to lie on different curves to maximize the regions of validity; see \cite[Chap. 11, Sect. 9.1]{Olver:1997:ASF}

Next we introduce functions that appear in the asymptotic expansions and their error bounds. Here and elsewhere the use of a circumflex (\textasciicircum) is in accord with the notation of \cite{Dunster:2017:COA} and \cite{Dunster:2020:LGE}, and is used to distinguish certain functions and paths that are defined in terms of $z$ rather than $\xi$. 

Let
\begin{equation}
\Phi(z) =\frac{4f(z) {f}^{\prime \prime }(z) -5{f}^{\prime 2}(z) }{16f^{3}(z) }+\frac{g(z)}{f(z)}.
\label{eq5}
\end{equation}
Then we define the set of coefficients 
\begin{equation}
\hat{F}_{1}(z) ={\tfrac{1}{2}}\Phi(z) ,\ \hat{F}_{2}(z) =-{\tfrac{1}{4}}f^{-1/2}(z) {\Phi }^{\prime
}(z),
\label{eq8aa}
\end{equation}
and 
\begin{equation}
\hat{F}_{s+1}(z) =-\tfrac{1}{2}f^{-1/2}(z) \hat{{F}
}_{s}^{\prime }(z) -\tfrac{1}{2}\sum\limits_{j=1}^{s-1}{\hat{F}%
_{j}(z) \hat{F}_{s-j}(z) }\ \left( {s=2,3,4\cdots }
\right).
\label{eq9}
\end{equation}
The odd coefficients appearing in the following asymptotic expansions are given by
\begin{equation}
\hat{E}_{2s+1}(z) =\int {\hat{F}_{2s+1}(z)
f^{1/2}(z) dz}\ \left( {s=0,1,2,\cdots }\right),
\label{eq7aa}
\end{equation}
where the integration constants must be chosen so that each $\left(
z-z_{0}\right) ^{1/2}\hat{E}_{2s+1}(z) $ is meromorphic at the turning point. As shown in \cite{Dunster:2017:COA}, the even ones can be determined without any integration, via the formal expansion
\begin{equation}
\sum\limits_{s=1}^{\infty }\dfrac{\hat{E}_{2s}(z) }{u^{2s}}
\sim -\frac{1}{2}\ln \left\{ 1+\sum\limits_{s=0}^{\infty }\dfrac{\hat{F}_{2s+1}(z) }{u^{2s+2}}\right\} +\sum\limits_{s=1}^{\infty }\dfrac{\alpha_{2s}}{u^{2s}},
\label{even}
\end{equation}
where each ${\alpha}_{2s}$ is arbitrarily chosen. These too are meromorphic at the turning point.

Next define two sequences $\left\{a_{s}\right\} _{s=1}^{\infty}$ and $\left\{\tilde{a}_{s}\right\} _{s=1}^{\infty}$ by $a_{1}=a_{2}=\frac{5}{72}$, $\tilde{a}_{1}=\tilde{a}_{2}=-{\frac{7}{72}}$, with subsequent terms $a_{s}$ and $\tilde{a}_{s}$ ($s=2,3,\cdots $) satisfying the same recursion formula 
\begin{equation}
b_{s+1}=\frac{1}{2}(s+1) b_{s}+\frac{1}{2}
\sum\limits_{j=1}^{s-1}{b_{j}b_{s-j}}.
\label{arec}
\end{equation}
Then let
\begin{equation}
\mathcal{E}_{s}(z) =\hat{E}_{s}(z) +
(-1)^{s}a_{s}s^{-1}\xi^{-s},
\label{eq40}
\end{equation}
and
\begin{equation}
\tilde{\mathcal{E}}_{s}(z) =\hat{E}_{s}(z) 
+(-1)^{s}\tilde{a}_{s}s^{-1}\xi^{-s}.
\label{eq38}
\end{equation}

Let $W(\zeta)=\mathrm{Ai}_{j}(\zeta):=\mathrm{Ai}\left(\zeta e^{-2\pi ij/3}\right)$ ($j=0,\pm 1$) be three numerically satisfactory solutions of the Airy equation $W''=\zeta W$. For large $|\zeta|$ these are characterized as being recessive for $\zeta \in \mathrm{\mathbf{S}}_{j}$ and dominant elsewhere. In particular as $u^{2/3}\zeta \rightarrow \infty$
\begin{equation}
\mathrm{Ai}\left( u^{2/3}\zeta \right) \sim \frac{e^{-u\xi }}{2\pi
^{1/2}u^{1/6}\zeta ^{1/4}}\quad \left(\left\vert \arg\left(u^{2/3}\zeta \right) \right\vert \leq \pi-\delta\right).
\label{Aiinfinity}
\end{equation}

Now from \cite{Dunster:2017:COA} we have three fundamental solutions of the homogeneous equation 
\begin{equation}
\label{odehomog}
d^{2}w/dz^{2}=\left\{u^{2}f(z)+g(z) \right\}w,
\end{equation}
given by
\begin{equation}
\label{eq47}
w_{m,j}(u,z)=\mathrm{Ai}_{j} \left(u^{2/3}\zeta \right)\mathcal{A}_{2m+2}(u,z)+\mathrm{Ai}'_{j} \left(u^{2/3}\zeta \right)\mathcal{B}_{2m+2} 
(u,z)\;(j=0,\pm 1),
\end{equation}
for $m=0,1,2,\cdots$.

For $z$ not too close to the turning point the coefficient functions are computable via the asymptotic expansions
\begin{multline}
\label{eq48}
 \mathcal{A}_{2m+2}(u,z)=\left\{\frac{\zeta }{f(z)} \right\}^{1/4}\exp 
\left\{\sum\limits_{s=1}^m \frac{\tilde{\mathcal{E}}_{2s}(z)}{u^{2s}} 
\right\}\cosh \left\{\sum\limits_{s=0}^m \frac{\tilde{\mathcal{E}}_{2s+1} 
(z)}{u^{2s+1}}  \right\} \\ 
 +\frac{1}{2}\left\{\frac{\zeta }{f(z)} \right\}^{1/4}\tilde{\varepsilon}_{2m+2}(u,z),
\end{multline}
and
\begin{multline}
\label{eq49}
 \mathcal{B}_{2m+2}(u,z)=\frac{1}{u^{1/3}\left\{\zeta f(z) 
\right\}^{1/4}}\exp \left\{\sum\limits_{s=1}^m \frac{{\mathcal{E}}_{2s}(z)}{u^{2s}} \right\}\sinh \left\{\sum\limits_{s=0}^{m} \frac{\mathcal{E}_{2s+1}(z)}{u^{2s+1}}  \right\} \\ 
+\frac{\varepsilon_{2m+2}(u,z)}{2u^{1/3}\left\{\zeta f(z) \right\}^{1/4}}.
\end{multline}

Let $B(z_{0},\delta)=\{z:|z-z_{0}| < \delta\}$ for arbitrary small $\delta>0$. Then, for $z$ lying in each of the three domains $Z_{j}(u)\cap Z_{k}(u) \setminus{B(z_{0},\delta)}$ ($j,k=0,\pm1, \, j<k$), bounds for the error terms $\tilde{\varepsilon}_{2m+2}(u,z)$ and $\varepsilon_{2m+2}(u,z)$ are furnished in \cite{Dunster:2020:SEB}, and they are $\mathcal{O}(u^{-2m-2})$ as $u \rightarrow \infty$ uniformly for $z$ lying in the specified domain. Note that for large $u$ the solution $w_{m,j}(u,z)$ is recessive (exponentially small) in $Z_{j}(u)$ and dominant outside this region.

The error terms in (\ref{eq48}) and (\ref{eq49}) become unbounded as $z \rightarrow z_{0}$. However, $\mathcal{A}_{2m+2}(u,z)$ and $\mathcal{B}_{2m+2}(u,z)$ are analytic at this point. Thus in a bounded neighbourhood of the turning point one can use the above expansions and Cauchy's integral formula to compute these functions. Specifically we have
\begin{multline}
\label{eq50}
 \mathcal{A}_{2m+2}(u,z)=\frac{1}{2\pi i}\oint_{\left| {t-z_{0} } 
\right|=r_{0} } {\exp \left\{\sum\limits_{s=1}^m \frac{\tilde{\mathcal{E}}_{2s}(t)}{u^{2s}}  \right\}} \\ 
 \times \cosh \left\{\sum\limits_{s=0}^m \frac{\tilde{\mathcal{E}}_{2s+1}(t)}{u^{2s+1}}  \right\}\left\{\frac{\zeta(t)}{f(t)} 
\right\}^{1/4}\frac{dt}{t-z}+\frac{1}{2}\tilde{\kappa }_{2m+2}(u,z),
\end{multline}
and
\begin{multline}
\label{eq51}
 \mathcal{B}_{2m+2}(u,z)=\frac{1}{2\pi iu^{1/3}}\oint_{\left| {t-z_{0} } 
\right|=r_{0} } {\exp \left\{\sum\limits_{s=1}^m \frac{{\mathcal{E}}_{2s}
(t)}{u^{2s}}  \right\}} \\ 
 \times \sinh \left\{\sum\limits_{s=0}^{m}\frac{\mathcal{E}_{2s+1}(t)}{u^{2s+1}} \right\}\frac{dt}{\left\{ f(t)\zeta(t) \right\}^{1/4}(t-z)}+\frac{\kappa_{2m+2}(u,z)}{2u^{1/3}},
\end{multline}
for some some suitably chosen $r_0>0$. Again full details and error bounds are given by \cite{Dunster:2020:SEB}. As shown in \cite{Dunster:2017:COA} these integrals are easy to compute to high accuracy. In particular the method overcomes the numerical difficulty of previous methods that involve cancellation of singularities.

We finally note from the connection formula \cite[Eq. 9.2.12]{NIST:DLMF}
\begin{equation}
\label{eq25}
\mathrm{Ai}(z)+e^{-2\pi i/3}\mathrm{Ai}_{1}(z)+e^{2\pi i/3}\mathrm{Ai}_{-1}(z)=0,
\end{equation}
and so from (\ref{eq47}) that
\begin{equation}
\label{eq52}
w_{m,0}(u,z)+e^{-2\pi i/3}w_{m,1}(u,z)+e^{2\pi i/3}w_{m,-1}(u,z)=0.
\end{equation}

\section{Particular solutions bounded away from turning point}
\label{away}

We shall consider the domains $Z_{j}(u)$ and $Z_{k}(u)$ ($j\neq k$) in pairs, and we define the three domains
\begin{equation}
Z^{(j,k)}(u):=Z_{j}(u)\cap Z_{k}(u) \quad (j,k=0,\pm1, \, j<k).
\end{equation}
We shall derive asymptotic expansions for three fundamental particular solutions of (\ref{eq1}), $w^{(j,k)}(u,z)$ ($j,k=0,\pm1$, $j < k$), which are slowly varying functions of $u$ in $Z^{(j,k)}(u)$. 

We assume that both reference points $z^{(j)}$ and $z^{(k)}$ lie in $Z^{(j,k)}(u)$. We take any branch of $\xi(z)$, provided it is a continuous function in this domain. Hence the unbounded values $\Re(u\xi^{(j)})$ and $\Re(u\xi^{(k)})$ are of opposite signs.  This means for each $z \in Z^{(j,k)}(u)$ there exists at least one progressive path in this domain, $\hat{\mathcal{L}}^{(j,k)}(z)$ say, connecting these two reference points and which contains $z$. Such a path depends on the argument of $u$ too, but for ease of notation we suppress this dependence.

Let $\Xi^{(j,k)}(u)$ and $\mathcal{L}^{(j,k)}(\xi)$ be the sets of points in the $\xi$ plane corresponding to $Z^{(j,k)}(u)$ and $\hat{\mathcal{L}}^{(j,k)}(z)$, respectively. Label the suffices so that $\Re(u\xi^{(j)})=-\infty $ and $\Re(u\xi^{(k)})=+\infty $, and denote
\begin{equation}
\psi(\xi)=\Phi(z(\xi)).
\label{phixi}
\end{equation}
Then, in terms of the variable $\xi$, the following result comes from \cite[Thm. 4.1]{Dunster:2020:LGE}.

\begin{theorem}
Let $\xi \in \Xi^{(j,k)}(u) \setminus{\{0\}} $, and $u $ be sufficiently large so that
$|u|>$ \linebreak $\frac{1}{2}\int_{\mathcal{L}^{(j,k)}(\xi) }{\left\vert {\psi(t) dt}\right\vert }$. Then there exists a unique particular solution of (\ref{eq1}) of the form 
\begin{equation}
G^{(j,k)}(u,\xi) =\frac{1}{f^{1/4}(z)}\left[\frac{1}{u^{2}}\sum
\limits_{s=0}^{n-1}\frac{G_{s}(\xi)}{u^{2s}}
+\varepsilon_{n}^{(j,k)}(u,\xi) \right],
\label{LG67}
\end{equation}
where
\begin{equation}
G_{0}(\xi) =-f^{-3/4}(z) p(z),
\label{LG63}
\end{equation}
\begin{equation}
G_{s+1}(\xi) =G_{s}^{\prime \prime }(\xi) -\psi
(\xi) G_{s}(\xi) \ (s=0,1,2,\cdots),
\label{LG68}
\end{equation}
and
\begin{multline}
\left\vert \varepsilon _{n}^{(j,k)}(u,\xi) \right\vert \leq \dfrac{1}{|u| ^{2n+2}}\left\{\left\vert G_{n}(\xi) \right\vert +\dfrac{1}{2}\int_{\mathcal{L}^{(j,k)}(\xi) }{\left\vert {{G}_{n}'(t) dt}\right\vert }\right\} \\
+\dfrac{L_{n}^{(j,k)}(\xi) }{2|u| ^{2n+3}}\left\{ {1-\dfrac{1}{2|u| }\int_{\mathcal{L}^{(j,k)}(\xi) }{\left\vert {\psi(t) dt}\right\vert }}\right\} ^{-1}\int_{\mathcal{L}^{(j,k)}(\xi) }\left\vert {\psi(t) dt}\right\vert ,
\label{LG82}
\end{multline}
in which
\begin{equation}
L_{n}^{(j,k)}(\xi) =\sup_{t\in \mathcal{L}^{(j,k)}(\xi) }\left\vert G_{n}(t) \right\vert +\frac{1}{2}\int_{\mathcal{L}^{(j,k)}(\xi) } \left\vert G_{n}'(t) dt \right\vert .
\label{LG80}
\end{equation}
In the bounds the integrals and supremum are assumed to exist.
\label{thm4}
\end{theorem}
\begin{remark}
The integrals certainly diverge at the turning point $\xi=0$. Therefore these expansions (and subsequent ones in this section) should not be used near the turning point, nor indeed near the boundaries of $\Xi^{(j,k)}(u)$. Instead the expansions of \cref{close} will be applicable, in conjunction with connection formulas that are provided therein.

We also mention that in order to obtain sharp bounds for large $\xi$ the paths of integration, here and throughout, should be chosen whenever possible so that $t^{-1}=\mathcal{O}(\xi ^{-1})$ for $t$ lying on the path in question.
\end{remark}

We shall convert this theorem to the variable $z$, but before doing so we present some extensions. Firstly, we give bounds on the derivative of the error term.

\begin{theorem}
Under the conditions of \cref{thm4}
\begin{multline}
\left\vert \frac{\partial \varepsilon _{n}^{(j,k)}(u,\xi )}{\partial \xi }
\right\vert \leq \frac{1}{\left\vert u^{2n+2}\right\vert }\left[ \left\vert
G_{n}^{\prime }(\xi)\right\vert +\frac{1}{2}\int_{\mathcal{L}^{(j,k)}(\xi)}
{\ \left\vert {G}_{n}^{\prime \prime }{(t)dt}\right\vert }\right. \\
\left. +\frac{1}{2}L_{n}^{(j,k)}(\xi)\left\{ {1-\frac{1}{2\left\vert
u\right\vert } \int_{\mathcal{L}^{(j,k)}(\xi)}{\left\vert {\psi(t)dt}
\right\vert }}\right\} ^{-1}\int_{\mathcal{L}^{(j,k)}(\xi)}\left\vert {\psi
(t) dt}\right\vert \right].
\label{LG100}
\end{multline}
\label{thmLGderivative} 
\label{thm5}
\end{theorem}

\begin{proof} 
Using variation of parameters \cite[Eq. (4.12)]{Dunster:2020:LGE} gives the integral equation
\begin{multline}
\varepsilon _{n}^{(j,k)}(u,\xi )=\dfrac{e^{-u\xi }}{2u}\int_{\xi^{(j)}}^{\xi
}{\ e^{ut}\left\{ {\dfrac{G_{n}(t)}{u^{2n}}-\psi(t)\varepsilon
_{n}^{(j,k)}(u,t)} \right\} dt} \\
+\dfrac{e^{u\xi }}{2u}\int_{\xi }^{\xi^{(k)}}e^{-ut}\left\{ {\dfrac{ G_{n}(t)
}{u^{2n}}-\psi(t) \varepsilon _{n}^{(j,k)}(u,t)}\right\} dt.  
\label{LG90}
\end{multline}
The paths of integration in both integrals are chosen to coincide with the
segments of $\mathcal{L}^{(j,k)}(\xi)$ from $\xi^{(j)}$ to $\xi$, and $\xi$ to $\xi^{(k)}$, respectively. We then integrate by parts twice the terms involving $u^{-2n}{e^{\pm ut}}
G_{n}(t)$ and we get 
\begin{multline}
\varepsilon _{n}^{(j,k)}(u,\xi )=\dfrac{G_{n}(\xi)}{u^{2n+2}}+\dfrac{
e^{-u\xi }}{2u} \int_{\xi^{(j)}}^{\xi }{e^{ut}\left\{ {\dfrac{{G}
_{n}^{\prime \prime }(t)}{u^{2n+2}}-\psi(t)\varepsilon _{n}^{(j,k)}(u,t)}
\right\} dt} \\
+\dfrac{e^{u\xi }}{2u}\int_{\xi }^ {\xi^{(k)}}e^{-ut}\left\{ {\dfrac{{G}
_{n}^{\prime \prime }(t)}{u^{2n+2}}-\psi(t)\varepsilon _{n}^{(j,k)}(u,t)}
\right\} dt.
\label{LG92}
\end{multline}
Hence on differentiation
\begin{multline}
\frac{\partial \varepsilon _{n}^{(j,k)}(u,\xi )}{\partial \xi }=\dfrac{
G_{n}^{\prime }(\xi)}{u^{2n+2}}-\dfrac{e^{-u\xi }}{2} \int_{\xi^{(j)}}^{\xi
}{\ e^{ut}\left\{ {\dfrac{{G}_{n}^{\prime \prime }(t)}{u^{2n+2}}-\psi
(t)\varepsilon _{n}^{(j,k)}(u,t)}\right\} dt} \\
+\dfrac{e^{u\xi }}{2}\int_{\xi } ^{\xi^{(k)}}e^{-ut}\left\{ {\dfrac{{G}
_{n}^{\prime \prime }(t)}{u^{2n+2}}-\psi(t)\varepsilon _{n}^{(j,k)}(u,t)}
\right\} dt.
\label{LG91}
\end{multline}
Thus by the monotonicity properties of the paths of integration we deduce that
\begin{multline}
\left\vert \frac{\partial \varepsilon _{n}^{(j,k)}(u,\xi )}{\partial \xi }
\right\vert \leq \left\vert \dfrac{G_{n}^{\prime }(\xi)}{u^{2n+2}}
\right\vert \\
+\frac{1}{2|u| ^{2n+2}}\int_{\mathcal{L}^{(j,k)}(\xi)}{
\left\vert {G}_{n}^{\prime \prime }{(t)dt}\right\vert }+\frac{1}{2}\int_{ \mathcal{L}^{(j,k}(\xi)}\left\vert \psi(t)\varepsilon _{n}^{(j,k)}(u,t)dt.
\right\vert  
\label{LG94}
\end{multline}
Now under the assumption that $|u|$ is sufficiently
large so that $\int_{\mathcal{L}^{(j,k)}(\xi)}{\left\vert {\psi(t)dt}
\right\vert } <2|u|$ we have from \cite[Eqs. (4.18) and (4.21)]{Dunster:2020:LGE}
\begin{multline}
\int_{\mathcal{L}^{(j,k)}(\xi)}{\left\vert {\psi(t)\varepsilon
_{n}^{(j,k)}(u,t)dt} \right\vert } \\
\leq \frac{L_{n}^{(j,k)}(\xi)}{|u| ^{2n+2}}\left\{ {1-
\frac{1}{2|u|}\int_{\mathcal{L}^{(j,k)}(\xi)}{
\left\vert {\ \psi(t)dt}\right\vert }}\right\} ^{-1}\int_{\mathcal{L}^{(j,k)}(\xi)}\left\vert {\ \psi(t) dt}\right\vert,
\label{LG96}
\end{multline}
and the result follows. 
\end{proof}

\begin{remark}
The proof of the bound (\ref{LG96}) uses the supremum of $|\varepsilon_{n}^{(j,k)}(u,t)|$ for $t\in \mathcal{L}^{(j,k)}(\xi)$ (see \cite[Eq. (4.17)]{Dunster:2020:LGE}). The existence of this supremum can be established from a suitably modified version of (\ref{LG90}) and \cite[Chap. 6, Thm. 10.1]{Olver:1997:ASF}.
\end{remark}

In the event the integrals involving the coefficients in the above error bounds diverge at one or both of the endpoints, for example if they are of exponential type $c$, the asymptotic expansions are still typically valid. As a simple example consider the equation
\begin{equation}
    \label{simpleODE}
    d^{2}w/d\xi^{2}-u^{2}w=e^{\xi}.
\end{equation}
The homogeneous solution has solutions $e^{\pm u \xi}$, and then by variation of parameters, or the method of undetermined coefficients, we obtain the slowly varying (in $u$) particular solution
\begin{equation}
    \label{psoln}
    w(\xi)=-\frac{e^{\xi}}{u^{2}-1}=
    -\frac{1}{u^{2}}\sum_{s=0}^{\infty}
    \frac{e^{\xi}}{u^{2s}} \quad (|u|>1),
\end{equation}
in accord with the asymptotic expansion given by \cref{thm4}.

To accommodate these situations it is often possible to modify \cref{thm4} and \cref{thm5}. For example, the former can be generalised in a straightforward manner as follows. Firstly, for fixed real $c$ let $\mathcal{L}_{c}^{(j,k)}(\xi)$ be an $R_{2}$ path containing a point $\xi \in \Xi\setminus{\{0\}}$ and as $t$ passes along the path from $\xi^{(j)}$ to $\xi^{(k)}$, the real parts of both $(u + c)t$ and $(u - c)t$ vary continuously and are monotonic. We denote $\Xi_{c}^{(j,k)}(u)$ to be the set of all such paths.

Before presenting our modified solutions, we note in passing the following.
\begin{proposition}
We have $\Xi_{c}^{(j,k)}(u) \subset \Xi^{(j,k)}(u)$.
\end{proposition}
\begin{proof}
Let $\xi \in \Xi_{c}^{(j,k)}(u)$ and $t_{1}$ and $t_{2}$ be any two points on $\mathcal{L}_{c}^{(j,k)}(\xi)$ satisfying $\Re\{(u\pm c)t_{1}\}\leq \Re\{(u\pm c)t_{2}\}$. Thus
\begin{equation}
\label{LG83a}
\Re(u t_{1}) - \Re(u t_{2})
\leq \pm c\Re(t_{2}-t_{1}).
\end{equation}
Since the RHS must have one of its values as negative or zero we deduce that $\Re(u t_{1}) \leq \Re(u t_{2})$. Hence $\mathcal{L}_{c}^{(j,k)}(\xi)$ satisfies the monotonicity condition of $\mathcal{L}^{(j,k)}(\xi)$, and therefore $\xi \in \Xi^{(j,k)}(u)$.
\end{proof}

\begin{theorem}
For fixed $c \in \mathbb{R}$ let $\xi \in \Xi_{c}^{(j,k)}(u) \setminus{\{0\}}$, and for $t \in \mathcal{L}_{c}^{(j,k)}(\xi)$ suppose $e^{-c t}G_{n}(t)=\mathcal{O}(1)$ and $\{ e^{-ct}G_{n}(t)\}^{\prime} =\mathcal{O}(t^{-\sigma})$ where $\sigma >1$. Then for
\begin{equation}
\max\{|c|,\tfrac{1}{2}\int_{\mathcal{L}_{c}^{(j,k)}(\xi)}{\left\vert{\psi(t)dt}\right\vert}\}<|u|<\infty,
\end{equation}
there exists a unique particular solution of (\ref{eq1}) 
of the form (\ref{LG67}), where
\begin{multline}
\left\vert \varepsilon_{n}^{(j,k)}(u,\xi )\right\vert \leq
\frac{1}{|u|^{2n+2}} \left\{\left\vert \dfrac{u^{2}
G_{n}(\xi)}{u^{2}-c^{2}}\right\vert 
+\frac{\left|u e^{c\xi }\right|}{2(|u|-|c|)}\int_{
\mathcal{L}_{c}^{(j,k)}(\xi)}{\left\vert {\left\{ {e^{-ct}G}_{n}(t)\right\}
^{\prime }dt}\right\vert }\right\}  \\
 +\dfrac{
\left|e^{c\xi }\right|L_{c,n}^{(j,k)}(u,\xi )}{2|u|^{2n+3}}\left\{ {1-\dfrac{1}{
2|u| }\int_{\mathcal{L}_{c}^{(j,k)}(\xi)}{
\left\vert {\psi(t)dt}\right\vert }}\right\} ^{-1}\int_{\mathcal{L}_{c}
^{(j,k)}(\xi)}\left\vert {\psi(t)dt}\right\vert,
\label{LG83}
\end{multline}
with 
\begin{multline}
L_{c,n}^{(j,k)}(u,\xi )=\left\vert \dfrac{u^{2}}{u^{2}-c^{2}}
\right\vert \sup_{t\in \mathcal{L}_{c}^{(j,k)}(\xi)}\left\vert {e^{-ct}}
G_{n}(t)\right\vert \\ +\frac{|u|}{2(|u|-|c|)}\int_{\mathcal{L}_{c}
^{(j,k)}(\xi)}{\left\vert \left\{ {e^{-ct}G}_{n}(t)\right\}^{\prime}dt \right\vert}.
\label{LG77}
\end{multline}
\label{thmLG2}
\end{theorem}
\begin{proof}
In (\ref{LG90}) we choose paths of integration in both integrals to lie on $\mathcal{L}_{c}^{(j,k)}(\xi)$. Let 
\begin{equation}
\varepsilon_{n}^{(j,k)}(u,\xi)=e^{c\xi}\varepsilon_{c,n}^{(j,k)}(u,\xi). 
\label{LG74}
\end{equation}
Consequently from (\ref{LG90}) we obtain
\begin{multline}
\varepsilon_{c,n}^{(j,k)}(u,\xi )=\dfrac{e^{-(u+c)\xi}}{2u}
\int_{\xi ^{(j)}}^{\xi }{e^{(u+c)t}\left\{ {\dfrac{{e^{-ct}}G_{n}(t)}{
u^{2n}}-\psi(t)\varepsilon_{c,n}^{(j,k)}(u,t)}\right\} dt} \\
+\dfrac{e^{(u-c)\xi} }{2u}\int_{\xi }^{\xi ^{(k)}}e^{-(u-c)t}
\left\{ {\dfrac{{e^{-ct}}G_{n}(t)}{u^{2n}}-\psi(t)
\varepsilon_{c,n}^{(j,k)}(u,t)}\right\} dt.  
\label{LG74b}
\end{multline}
Next integration by parts the terms involving ${e^{-ct}}G_{n}(t)$ gives
\begin{multline}
\varepsilon_{c,n}^{(j,k)}(u,\xi )=\dfrac{{e^{-c\xi }}G_{n}(\xi)}{
u^{2n}\left(u^{2}-c^{2}\right) } \\
-\dfrac{e^{-(u+c)\xi }}{2u}\int_{\xi
^{(j)}}^{\xi }{e^{(u+c)t}\left\{ {\dfrac{\left\{ {e^{-ct}G}_{n}(t)\right\}
^{\prime }}{u^{2n}(u+c)}+\psi(t)\varepsilon_{c,n}^{(j,k)}(u,t)}
\right\} dt} \\
+\dfrac{e^{(u-c)\xi }}{2u}\int_{\xi }^{\xi ^{(k)}}e^{-(u-c)t}
\left\{ {\dfrac{\left\{ {e^{-ct}G}_{n}(t)\right\} ^{\prime }}{u^{2n}(u-c)}
-\psi(t)\varepsilon_{c,n}^{(j,k)}(u,t)}\right\} dt, 
\label{LG73}
\end{multline}
and hence from the monotonicity condition on the paths of integration we see that 
\begin{multline}
\left\vert {\varepsilon_{c,n}^{(j,k)}(u,\xi )}\right\vert \leq
\left\vert \dfrac{{e^{-c\xi}}G_{n}(\xi)}{u^{2n}\left(u^{2}-c^{2}\right) }
\right\vert +\frac{1}{2|u| ^{2n+1}(|u|-|c|)}\int_{
\mathcal{L}_{c}^{(j,k)}(\xi)}{\left\vert {\left\{ {e^{-ct}G}
_{n}(t)\right\} ^{\prime }dt}\right\vert } \\
+\frac{1}{2|u|}\int_{\mathcal{L}_{c}^{(j,k)}(\xi
)}\left\vert {\psi(t)\varepsilon_{c,n}^{(j,k)}(u,t)dt}\right\vert.
\label{LG75}
\end{multline}
Then following a similar step to \cite[Eqs. (4.18) and (4.21)]{Dunster:2020:LGE}, and using (\ref{LG77}), yields the bound 
\begin{multline}
\int_{\mathcal{L}_{c}^{(j,k)}(\xi)}\left\vert \psi(t)\varepsilon_{c,n}^{(j,k)}(u,t)dt\right\vert  \\
\leq \frac{L_{c,n}^{(j,k)}(u,\xi )}{|u|^{2n+2}}
\left\{1-\dfrac{1}{2|u| }\int_{\mathcal{L}_{c}
^{(j,k)}(\xi)}{\left\vert\psi(t)dt\right\vert }\right\} ^{-1}\int_{\mathcal{L}_{c}^{(j,k)}(\xi)}\left\vert\psi(t)dt\right\vert,
\label{LG78}
\end{multline}
where $L_{c,n}^{(j,k)}(u,\xi )$ is given by (\ref{LG77}). On combining (\ref{LG74}), (\ref{LG75}) and (\ref{LG78}) we arrive at (\ref{LG83}). 
\end{proof}

Similarly if $G_{n}(\xi)=\mathcal{O}\left( {\xi ^{\sigma }}\right)$ and $G_{n}^{\prime }(\xi)=\mathcal{O}\left( \xi ^{\sigma -1}\right)$ ($\sigma \geq 0$) as $\xi \rightarrow \xi^{(j)}$ and $\xi \rightarrow \xi^{(k)}$, one can prove that 
\begin{multline}
\left\vert \varepsilon _{n}^{(j,k)}(u,\xi )\right\vert \leq
 \frac{1}{|u|^{2n+2}} \left\{\left\vert G_{n}(\xi)\right\vert +\dfrac{\left| \xi ^{\sigma +1} \right|}{2}\int_{\breve{\mathcal{L}}_{\sigma}^{(j,k)}(\xi)}{\left\vert t^{-\sigma -1}{G}_{n}^{\prime }(t){dt}\right\vert }\right\}  \\
 +\dfrac{
\left| \xi ^{\sigma +1} \right| \breve{L}_{\sigma,n}^{(j,k)}(u,\xi )}{2|u|^{2n+3}}\left\{ {1-\dfrac{1}{
2|u| }\int_{\breve{\mathcal{L}}_{\sigma}^{(j,k)}(\xi)}{
\left\vert {\psi(t)dt}\right\vert }}\right\} ^{-1}\int_{\breve{\mathcal{L}}_{\sigma}^{(j,k)}(\xi)}\left\vert {\psi(t)dt}\right\vert,  
\label{LG87a}
\end{multline}
where
\begin{equation}
\breve{L}_{\sigma,n}^{(j,k)}(\xi)=\sup_{t\in \breve{\mathcal{L}}_{\sigma}^{(j,k)}(\xi)}\left\vert {t}
^{-\sigma-1}G_{n}(t)\right\vert +\frac{1}{2}\int_{\breve{\mathcal{L}}_{\sigma}^{(j,k)}(\xi)}{\left\vert t^{-\sigma-1}{G}_{n}^{\prime}(t){dt}\right\vert}.  
\label{LG88}
\end{equation}
This time the path must meet the modified monotonicity condition: (ii)$'$ as $t$ passes along $\breve{\mathcal{L}}_{\sigma}^{(j,k)}(\xi)$ from $\xi^{(j)}$ to $\xi^{(k)}$, the real parts of both $ut + (\sigma+1)\ln(t)$ and $ut - (\sigma+1)\ln(t)$ are monotonic.

In terms of $z$ we write $w^{(j,k)}(u,z)= G^{(j,k)}(u,\xi(z))$ and can express \cref{thm4} in the following form.
\begin{theorem}
\label{thm6}
Under the conditions of \cref{thm4} and for $z \in Z^{(j,k)}(u) \setminus{\{z_{0}\}}$ there exists a unique particular solution of (\ref{eq1}) of the form 
\label{gz}
\begin{equation}
\label{eq42}
w^{(j,k)}(u,z)=\frac{1}{u^{2}}\sum\limits_{s=0}^{n-1} \frac{\hat{G}_{s}(z)}{u^{2s}} +\hat{\varepsilon}_{n}^{(j,k)}(u,z),
\end{equation}
where
\begin{equation}
\label{eq43}
\hat{G}_{0}(z)=-p(z)/f(z),
\end{equation}
\begin{equation}
\label{eq44}
\hat{G}_{s+1}(z)=\frac{\hat{G}''_{s}(z)-g(z)\hat{G}_{s} 
(z)}{f(z)}\quad (s=0,1,2,\cdots ),
\end{equation}
\begin{multline}
\label{eq45}
 \left|\hat{\varepsilon}_{n}^{(j,k)}(u,z) \right|
 \le 
\frac{1}{\vert u\vert^{2n+2}}\left\{ \left| \hat{G}_{n}(z)\right|
+\frac{1}{2 |f(z)|^{1/4}}
\int_{\hat{\mathcal{L}}^{(j,k)}(z)} {
\left| 
\left\{f^{1/4}(t) \hat{G}_{n}(t)\right\}'dt
\right|}  \right\} \\ 
 +\frac{\hat{{L}}_{n}^{(j,k)}(z)}{2 |u|^{2n+3}
 {|f(z)|}^{1/4} }\left\{ 
{1-\frac{1}{2\vert u\vert }\int_{\hat{\mathcal{L}}^{(j,k)}(z)} {\left| 
{f^{1/2}(t)\Phi(t)dt} \right|} } \right\}^{-1}\int_{\hat{\mathcal{L}}^{(j,k)}(z)} {\left| f^{1/2}(t)\Phi(t) dt \right|},
\end{multline}
in which
\begin{equation}
\label{eq46}
\hat{{L}}_{n}^{(j,k)}(z)=\sup _{t\in \hat{\mathcal{L}}^{(j,k)}(z)}
\left|f^{1/4}(t) \hat{G}_{n}(t)  \right|+\frac{1}{2}\int_{\hat{\mathcal{L}}^{(j,k)}(z)} \left| 
\left\{f^{1/4}(t)\hat{G}_{n}(t)\right\}'dt
\right|.
\end{equation}
\end{theorem}

We end this section by observing that the expansion (\ref{eq42}) can be differentiated to give an approximation for the derivatives of the solutions. To do so, we have from (\ref{xi}) and (\ref{LG67}) that $d\xi/dz=f^{1/2}(z)$ and
\begin{equation}
\hat{\varepsilon}_{n}^{(j,k)}(u,z)
=f^{-1/4}(z)\varepsilon_{n}^{(j,k)}(u,\xi).
\label{eq46b}
\end{equation}
Thus from differentiating (\ref{eq42}) we obtain
\begin{equation}
\frac{\partial w^{(j,k)}(u,z)}{\partial z}=\frac{1}{u^{2}}\sum\limits_{s=0}^{n-1} \frac{{\hat{G}_{s}}'(z)}{u^{2s}} 
-\frac{f'(z)\varepsilon _{n}^{(j,k)}(u,\xi )}{4f^{5/4}(z)}
+f^{1/4}(z) \frac{\partial 
\varepsilon _{n}^{(j,k)}(u,\xi )}{\partial \xi }
\label{eq46c}
\end{equation}
and in this bounds on the error terms are supplied by \cref{thm4} and \cref{thm5}.

\section{Scorer functions} 
\label{scorer}

The Scorer function $\mathrm{Hi}(z)$ plays an important role in our subsequent expansions, and is defined by
\begin{equation}
\label{eq2}
\mathrm{Hi}(z)=\frac{1}{\pi }\int_{0}^{\infty} \exp  \left(-\tfrac{1}{3} t^{3}+z t \right)dt.
\end{equation}
It is also the uniquely defined particular solution of the inhomogeneous Airy equation
\begin{equation}
\label{eq3}
\frac{d^{2}w}{dz^{2}}-zw=\frac{1}{\pi},
\end{equation}
having the behaviour
\begin{equation}
\label{eq2a}
\mathrm{Hi}(z)\sim -\frac{1}{\pi z} \quad \left(z\to \infty,\; \left| \arg (-z) \right|\le \tfrac{2}{3}\pi -\delta \right),
\end{equation}
for arbitrary small positive $\delta$. Its uniqueness is by virtue of all other particular solutions being exponentially large in part, or whole, of the specified sector. This is easily seen by noting that for arbitrary constants $c_{\pm 1}$
\begin{equation}
\label{eq3a}
w(z)=\mathrm{Hi}(z) +c_{-1}\mathrm{Ai}_{-1}(z)
+c_{1}\mathrm{Ai}_{1}(z),
\end{equation}
is the general solution of (\ref{eq3}). Recalling that $\mathrm{Ai}_{\pm 1}(z)$ is exponentially small in $\mathrm{\mathbf{S}}_{\pm 1}$ (as defined by (\ref{eq6}) with $\zeta=z$) and exponentially large in $\mathrm{\mathbf{S}}_{\mp 1}$, we see that the function (\ref{eq3a}) can only be bounded in $\mathrm{\mathbf{S}}_{-1} \cup \mathrm{\mathbf{S}}_{1}$ if $c_{-1}=c_{1}=0$.

From variation of parameters on (\ref{eq3}) we find the Scorer function also has the useful contour integral representation
\begin{equation}
\label{eq4}
\mathrm{Hi}(z)=2i\left\{\mathrm{Ai}_{-1}(z)\int_{\infty \exp (2\pi i/3)}^{z} {\mathrm{Ai}_{1}(t)dt} -\mathrm{Ai}_{1} \left( z 
\right)\int_{\infty \exp (-2\pi i/3)}^{z} {\mathrm{Ai}_{-1}(t)dt} \right\}.
\end{equation}
In deriving this we used (\ref{Aiinfinity}), (\ref{eq2a}) and the Wronskian \cite[Eq. 9.2.9]{NIST:DLMF}
\begin{equation}
\label{eq70}
\mathscr{W}\left\{ {\mathrm{Ai}_{1}(z),\mathrm{Ai}_{-1}(z)} \right\}=\frac{1}{2\pi i}.
\end{equation}
We also record here an accompanying Wronskian \cite[Eq. 9.2.8]{NIST:DLMF} which we shall use later, viz.
\begin{equation}
\label{eq71}
\mathscr{W}\left\{ {\mathrm{Ai}(z),\mathrm{Ai}_{\pm 1}(z)} \right\}=\frac{e^{\pm \pi i/6}}{2\pi}.
\end{equation}

The extension of (\ref{eq2a}) (and for the derivative of the function) to an asymptotic expansion is well known \cite[Sect. 9.12(viii)]{NIST:DLMF}. Here we give simple error bounds, which somewhat surprisingly appear to be new. In this we use the superscript $(-1,1)$ on the error terms because $\mathrm{Hi}(z)$ is bounded for $z \in \mathbf{S}_{-1} \cup \mathbf{S}_{1}$. See also (\ref{eq22a}) - (\ref{eq22c}) below.
\begin{theorem}
\label{thm:Hi}
Let $\theta=\arg(-z)$ and $\delta$ be an arbitrary small positive number. Then for integer $n\geq 0$
\begin{equation}
\label{eq7}
\mathrm{Hi}(z)=-\frac{1}{\pi z}\sum\limits_{k=0}^n 
\frac{(3k)!}{k!(3z^{3})^{k}} +\frac{1}{\pi }\varepsilon_{2n+4}^{(-1,1)}(z),
\end{equation}
where
\begin{equation}
\label{eq13}
\left| {\varepsilon_{2n+4}^{(-1,1)}(z)} \right|\le 
\frac{(3n+3)!}{3^{n+1}(n+1)!|z|^{3n+4}}
\quad \left( \vert \theta \vert \le 
\tfrac{1}{6}\pi \right),
\end{equation}
and
\begin{equation}
\label{eq15c}
 \left| {\varepsilon_{2n+4}^{(-1,1)}(z)} \right|\le 
\frac{(3n+3)!}{3^{n+1}(n+1)!\left\{ {|z| \cos \left( {\vert 
\theta \vert -\tfrac{1}{6}\pi } \right)} \right\}^{3n+4}}
\quad \left( \tfrac{1}{6}\pi <\vert \theta \vert \le \tfrac{2}{3}\pi -\delta  \right).
\end{equation}
Likewise
\begin{equation}
\label{eq17}
\mathrm{Hi}'(z)=\frac{1}{\pi z^{2}}\sum\limits_{k=0}^n \frac{(3k+1)!}{k!(3z^{3})^{k}} +\frac{1}{\pi }\tilde{{\varepsilon}}_{2n+5}^{(1,-1)}(z),
\end{equation}
where
\begin{equation}
\label{eq19}
\left| \tilde{\varepsilon }_{2n+5}^{(-1,1)}(z)
\right|\le 
\frac{(3n+4)!}{3^{n+1}(n+1)!|z|^{3n+5}}\quad \left( 
| \theta | \le \tfrac{1}{6}\pi  \right),
\end{equation}
and
\begin{equation}
\label{eq20}
\left| \tilde{\varepsilon}_{2n+5}^{(-1,1)}(z) \right|\le \frac{(3n+4)!}{3^{n+1}(n+1)!\left\{ {|z| \cos \left(|\theta| -\tfrac{1}{6}\pi \right)} \right\}^{3n+5}}\quad \left( \tfrac{1}{6}\pi <
| \theta | \le \tfrac{2}{3}\pi 
-\delta  \right).
\end{equation}
\end{theorem}

\begin{proof}
We use the Maclaurin polynomial with remainder term \cite[Sect. 1.4(vi)]{NIST:DLMF}
\begin{equation}
\label{eq5aa}
\exp \left(-\tfrac{1}{3}t^{3} \right)=\sum\limits_{k=0}^n
{(-1)^{k}\frac{t^{3k}}{3^{k}k!}} +R_{3n+3}(t),
\end{equation}
where
\begin{equation}
\label{eq6aa}
R_{3n+3}(t)=\frac{(-1)^{n+1}}{n!}\int_{0}^{t^{3}/3} \left( \tfrac{1}{3}t^{3}-v \right)^{n}e^{-v} dv.
\end{equation}
Although the form (\ref{eq6aa}) for the remainder is typically associated with real argument, it easily verified here to be valid for all complex $t$, and we take the path of integration to be a straight line.

Plugging (\ref{eq5aa}) into (\ref{eq2}) we obtain, by integrating term by term, the series (\ref{eq7}) with remainder
\begin{equation}
\label{eq8}
\varepsilon_{2n+4}^{(-1,1)}(z)=\int_{0}^{\infty} {R_{3n+3}(t)e^{zt}} dt.
\end{equation}
Next let $\phi =\arg(t^{3})$, and then by change of variable $v \rightarrow ve^{i\phi }$ we have
\begin{multline}
\label{eq9aa}
\int_{0}^{t^{3}/3} {\left( {\tfrac{1}{3}t^{3}-v} \right)^{n}e^{-v}} dv=e^{i\phi }\int_{0}^{t^{3}\exp (-i\phi )/3} {\left( \tfrac{1}{3}t^{3}-ve^{i\phi } \right)^{n}\exp \left( {-ve^{i\phi }} \right)} dv 
 \\ =e^{i(n+1)\phi }\int_{0}^{\vert t\vert^{3}/3}
 {\left( \tfrac{1}{3}\vert t\vert^{3}-v \right)^{n}\exp \left( {-ve^{i\phi }} \right)} dv.
\end{multline}
Now if $|\phi| \le \frac{1}{2}\pi $ we have $|\exp(-ve^{i\phi })|\le 1$ for $v$ real and non-negative, and hence from (\ref{eq6aa}) and (\ref{eq9aa}) we deduce that
\begin{equation}
\label{eq10}
\left| {R_{3n+3}(t)} \right|\le \frac{1}{n!}\int_{0}^{\vert t\vert^{3}/3} 
\left( {\tfrac{1}{3}\vert t\vert^{3}-v} \right)^{n} dv,
\end{equation}
which on integration yields
\begin{equation}
\label{eq11}
\left| {R_{3n+3}(t)} \right|\le \frac{\vert t\vert 
^{3n+3}}{3^{n+1}(n+1)!}\quad \left( {\vert \arg(t)\vert \le \tfrac{1}{6}\pi } \right).
\end{equation}

First assume that $|\theta| \le 
\frac{1}{6}\pi$. Then by deforming the contour of (\ref{eq8}) to the line $t\exp(-i\theta )$ where $0\le t<\infty $, we have $zt\exp(-i\theta )=-|z|t$ and hence
\begin{equation}
\label{eq12}
\varepsilon_{2n+4}^{(-1,1)}(z)=e^{-i\theta }\int_{0}^{\infty} {R_{3n+3} 
\left(te^{-i\theta } \right)e^{-|z| t}} dt.
\end{equation}
Thus from (\ref{eq8}), (\ref{eq11}) and (\ref{eq12}) we get
\begin{equation}
\label{eq13aaa}
\left| {\varepsilon_{2n+4}^{(-1,1)}(z)} \right|\le 
\frac{1}{3^{n+1}(n+1)!}\int_{0}^{\infty} {t^{3n+3}e^{-|z| t}} dt,
\end{equation}
which on using
\begin{equation}
\label{eq13bb}
\int_{0}^{\infty} {t^{n}e^{-a t}} dt
=\frac{n!}{a^{n+1}} 
\quad \left(n=0,1,2,\cdots, \, \Re(a)>0\right),
\end{equation}
yields (\ref{eq13}).

Next for $\frac{1}{6}\pi <\vert \theta \vert \le \frac{2}{3}\pi -\delta $ deform the contour of (\ref{eq8}) to the line $t\exp (\mp i\pi /6)$ (for $\theta \gtrless 0$) where $0\le t<\infty $. 
We then have
\begin{equation}
\label{eq14}
\varepsilon_{2n+4}^{(-1,1)}(z) =e^{\mp i\pi /6}\int_{0}^{\infty} {R_{3n+3} 
\left( {te^{\mp i\pi /6}} \right)\exp \left\{ {-|z| t\exp \left( {\pm i\left| {\theta \mp \tfrac{1}{6}\pi } \right|} \right)} 
\right\}} dt,
\end{equation}
and therefore from (\ref{eq11})
\begin{equation}
\label{eq15}
 \left| {\varepsilon_{2n+4}^{(-1,1)}(z)} \right|\le 
\frac{1}{3^{n+1}(n+1)!}\int_{0}^{\infty} {t^{3n+3}\exp \left\{ {-|z| t\cos \left(|\theta| -\tfrac{1}{6}\pi\right)} 
\right\}} dt,
\end{equation}
again with $t$ real. The bound (\ref{eq15c}) then follows from (\ref{eq13bb}).

Finally on differentiating (\ref{eq2}) we have
\begin{equation}
\label{eq16}
\mathrm{Hi}'(z)=\frac{1}{\pi }\int_{0}^{\infty} {t\exp } \left( -\tfrac{1}{3}t^{3}+zt \right)dt,
\end{equation}
and so from inserting (\ref{eq5aa}) we get (\ref{eq17}) where
\begin{equation}
\label{eq18}
\tilde{\varepsilon }_{2n+5}^{(-1,1)}(z)=\int_{0}^{\infty} {t R_{3n+3}(t)e^{zt}} dt.
\end{equation}
The bounds (\ref{eq19}) and (\ref{eq20}) then follow from this in a similar manner to the derivation of (\ref{eq13}) and (\ref{eq15c}).
\end{proof}

Now $\mathrm{Hi}(z)$ is real-valued but exponentially large for positive $z=x$. The function $w(x)=-\mathrm{Gi}(x)$ is a solution of (\ref{eq3}) that is real-valued and bounded as $x \rightarrow \infty$, and hence is usually preferred in this situation. It is defined by
\begin{equation}
\label{Gi}
\mathrm{Gi}(x)=\frac{1}{\pi }\int_{0}^{\infty} \sin  \left(\tfrac{1}{3} t^{3}+x t \right)dt \quad (x \in \mathbb{R}).
\end{equation}

Although we will not make use of this function in this paper it is useful in applications, and so we give error bounds for its asymptotic approximation for large positive $x$. To do so we use, for complex $z$, the identity
\begin{equation}
\label{eq30}
\mathrm{Gi}(z)=\tfrac{1}{2}e^{\pi i/3}\mathrm{Hi}\left( 
{ze^{-2\pi i/3}} \right)+\tfrac{1}{2}e^{-\pi i/3}\mathrm{Hi}\left( 
{ze^{2\pi i/3}} \right).
\end{equation}
Hence from (\ref{eq7})
\begin{multline}
\label{eq31}
\mathrm{Gi}(z)=\frac{1}{\pi z}\sum\limits_{k=0}^n \frac{(3k)!}{k!\left( 
{3z^{3}} \right)^{k}} \\
+\frac{e^{\pi i/3}}{2\pi }\varepsilon_{2n+4}^{(-1,1)} \left( ze^{-2\pi i/3} \right)
+\frac{e^{-\pi i/3}}{2\pi 
}\varepsilon_{2n+4}^{(-1,1)} \left( {ze^{2\pi i/3}} \right),
\end{multline}
and similarly for its derivative.

From this we can apply \cref{thm:Hi} to obtain error bounds, valid for $\arg(z) \leq \tfrac{1}{3}\pi -\delta$. We present the case for positive argument only.

\begin{theorem}
For $x>0$
\begin{equation}
\label{eq31a}
\mathrm{Gi}(x)=\frac{1}{\pi x}\sum\limits_{k=0}^n \frac{(3k)!}{k!\left( 
{3x^{3}} \right)^{k}} 
+\frac{1}{\pi}\varepsilon_{2n+4}^{(\mathrm{Gi})}(x),
\end{equation}
and
\begin{equation}
\label{eq32}
\mathrm{Gi}'(x)=-\frac{1}{\pi x^{2}}\sum\limits_{k=0}^n
\frac{(3k+1)!}{k!\left( {3x^{3}} \right)^{k}} +\frac{1}{\pi}\tilde{\varepsilon}_{2n+5}^{(\mathrm{Gi})}(x),
\end{equation}
where
\begin{equation}
\label{eq31b}
 \left| \varepsilon_{2n+4}^{(\mathrm{Gi})}(x) \right|\le
\frac{(3n+3)!}{3^{(5n/2)+3}(n+1)!\left(\tfrac{1}{2}x\right)^{3n+4}},
\end{equation}
and
\begin{equation}
\label{eq32a}
\left| \tilde{\varepsilon}_{2n+5}^{(\mathrm{Gi})}(x)
\right|
\leq
\frac{(3n+4)!}{3^{(5n+7)/2}(n+1)!\left(\tfrac{1}{2}x  \right)^{3n+5}}.
\end{equation}
\end{theorem}

We finish this section by defining three numerically satisfactory scaled Scorer functions that will appear in our asymptotic solutions of (\ref{eq1}) valid near the turning point. These are given by
\begin{equation}
\label{eq22a}
 \mathrm{Wi}^{(-1,1)}(z)=\pi \mathrm{Hi}(z),
\end{equation}
 \begin{equation}
\label{eq22b}
 \mathrm{Wi}^{(0,1)}(z)=\pi e^{-2\pi i/3}\mathrm{Hi}\left(ze^{-2\pi i/3}\right),
\end{equation}
and
\begin{equation}
\label{eq22c}
 \mathrm{Wi}^{(-1,0)}(z)=\pi e^{2\pi i/3}\mathrm{Hi}\left(ze^{2\pi i/3} 
\right).
\end{equation}

From \cite[Eq. 9.12.14]{NIST:DLMF}
\begin{equation}
\label{eq21}
\mathrm{Hi}(z)=e^{\pm 2\pi i/3}\mathrm{Hi}\left( {ze^{\pm 2\pi i/3}} \right)+2e^{\mp \pi i/6}\mathrm{Ai}_{\pm 1}(z),
\end{equation}
we get connection formulas
\begin{equation}
\label{eq23}
\mathrm{Wi}^{(-1,1)}(z)=\mathrm{Wi}^{(0,1)}(z)+2\pi e^{-\pi i/6}\mathrm{Ai}_{1}(z),
\end{equation}
and
\begin{equation}
\label{eq24}
\mathrm{Wi}^{(-1,1)}(z)=\mathrm{Wi}^{(-1,0)}(z)+2\pi e^{\pi i/6}\mathrm{Ai}_{-1}(z).
\end{equation}
From these and (\ref{eq25}) we also have the relation
\begin{equation}
\label{eq26}
\mathrm{Wi}^{(0,1)}(z)=\mathrm{Wi}^{(-1,0)}(z)+2\pi i\mathrm{Ai}(z).
\end{equation}

The significance of these three functions is that each $\mathrm{Wi}^{(j,k)}(u^{2/3}\zeta)$ is a particular solution of \footnote{In \cite{Dunster:2020:ASI} the RHS $1$ should be $u^{4/3}$}
\begin{equation}
\label{eq27}
d^{2}W/d\zeta^{2}-u^{2}\zeta W=u^{4/3},
\end{equation}
having the unique behaviour of being bounded in $\mathbf{T}_{j}\cup \mathbf{T}_{k}$ as $u \rightarrow \infty$. This can be seen from their definitions, (\ref{eq3}), and \cref{thm:Hi} which yields
\begin{equation}
\label{eq28}
\mathrm{Wi}^{(j,k)}(u^{2/3}\zeta) = -\frac{1}{u^{2/3}\zeta}\sum\limits_{k=0}^{n} 
\frac{(3k)!}{k!(3u^{2}\zeta^{3})^{k}}
+\varepsilon_{2n+4}^{(j,k)}(u^{2/3}\zeta),
\end{equation}
and
\begin{equation}
\label{eq29}
{\mathrm{Wi}^{(j,k)}}'(u^{2/3}\zeta)= \frac{1}{u^{4/3}\zeta^{2}}\sum\limits_{k=0}^{n} 
\frac{(3k+1)!}{k!(3u^{2}\zeta^{3})^{k}}
+\tilde{\varepsilon}_{2n+5}^{(j,k)}(u^{2/3}\zeta),
\end{equation}
where
\begin{equation}
\label{eq28a}
\varepsilon_{2n+4}^{(j,k)}(z)
=\varepsilon_{2n+4}^{(-1,1)}\left(ze^{-2(j+k)\pi i/3}\right),
\end{equation}
and
\begin{equation}
\label{eq28b}
\tilde{\varepsilon}_{2n+5}^{(j,k)}(z)
=\tilde{\varepsilon}_{2n+5}^{(-1,1)}\left(ze^{-2(j+k)\pi i/3}\right).
\end{equation}
The error terms satisfy the bounds (\ref{eq13}) and (\ref{eq15c}) with $z=u^{2/3}\zeta e^{-2(j+k)\pi i/3}$, and hence for large $u^{2/3}\zeta$ are $\mathcal{O}\{(u^{2/3}\zeta)^{-2n-4}\}$ and $\mathcal{O}\{(u^{2/3}\zeta)^{-2n-5}\}$, respectively, in $\mathbf{T}_{j}\cup \mathbf{T}_{k}$ (except near the boundary of this domain).

The (exponentially large) asymptotic behaviour of $\mathrm{Wi}^{(j,k)}(u^{2/3}\zeta)$ for $\zeta \in \mathbf{T}_{l}$ ($l\neq j,k$) comes from (\ref{Aiinfinity}), an appropriate choice from the connection formulas (\ref{eq23}) - (\ref{eq26}), along with (\ref{eq28}). Similarly for the derivatives.

\section{Solutions close to turning point} 
\label{close}
In \cref{away} the three solutions of (\ref{eq1}) $w^{(j,k)}(u,z)$ were defined which involve simple expansions and error bounds (\cref{thm6}), and have the uniquely-defining behaviour of being bounded in $Z^{(j,k)}(u)$. Since these asymptotic expansions break down near $z=z_{0}$ we now derive approximations for them that are valid in an unbounded domain that contains this turning point, and these will involve the Scorer functions defined in the previous section.

We do this by utilising certain connections formulas, the first of which comes from noticing the difference of any two particular solutions of (\ref{eq1}) is a solution of the corresponding homogeneous equation (\ref{odehomog}). Hence we can assert that
\begin{equation}
\label{eq53}
w^{(-1,1)}(u,z)-w^{(0,1)}(u,z)=c_{m,0}(u)w_{m,0}(u,z)+c_{m,1}(u)w_{m,1}(u,z),
\end{equation}
for some constants $c_{m,0}(u)$ and $c_{m,1}(u)$, and where $w_{m,j}(u,z)$ ($j=0,1$) are the homogeneous solutions given by (\ref{eq47}).
On letting $z\rightarrow z^{(1)}$ we see that all functions in (\ref{eq53}) are bounded, with the exception of $w_{m,0}(u,z)$, which implies $c_{m,0}(u)=0$. For later convenience we write $c_{m,1}(u)=2\pi e^{-\pi i/6}\gamma_{m}(u)$ and hence can express (\ref{eq53}) in the form
\begin{equation}
\label{eq54}
w^{(-1,1)}(u,z)=w^{(0,1)}(u,z)+2\pi e^{-\pi i/6}\gamma_{m}(u)w_{m,1}(u,z).
\end{equation}

With this definition of the connection coefficient $\gamma_{m}(u)$ we can state our main result. In this we define
\begin{equation}
\label{eq54a}
Z^{(\mathcal{G})}(u):=Z^{(-1,1)}(u)\cup Z^{(0,1)}(u)\cup Z^{(-1,0)}(u).
\end{equation}
Note this contains all three reference point singularities $z^{(j)}$ ($j=0,\pm1$) as well as a full neighbourhood of the turning point $z_{0}$.
\begin{theorem}
\label{ThmG}
There exists a function $\mathcal{G}_{m}(u,z)$ which analytic in $Z$ and bounded in $Z^{(\mathcal{G})}(u)$, such that the three fundamental solutions given by \cref{thm6} can be expressed in the form
\begin{multline}
\label{eq55}
 w^{(j,k)}(u,z)=\gamma_{m}(u)\left\{\mathrm{Wi}^{(j,k)}\left(u^{2/3}\zeta  \right)\mathcal{A}_{2m+2}(u,z) \right. \\ + \left. {\mathrm{Wi}^{(j,k)}}^{\prime}\left(u^{2/3}\zeta \right)\mathcal{B}_{2m+2}(u,z) \right\} +\mathcal{G}_{m}(u,z) \quad (j,k=0,\pm1, \, j<k).
\end{multline}
Moreover
\begin{multline}
\label{eq55e}
 \partial w^{(j,k)}(u,z) / \partial z
 =\gamma_{m}(u)\left\{\mathrm{Wi}^{(j,k)}\left(u^{2/3}\zeta  \right)\mathcal{C}_{2m+2}(u,z) \right. \\ + \left. {\mathrm{Wi}^{(j,k)}}^{\prime}\left(u^{2/3}\zeta \right)\mathcal{D}_{2m+2}(u,z) \right\} +\mathcal{H}_{m}(u,z) \quad (j,k=0,\pm1, \, j<k),
\end{multline}
where
\begin{equation}
\label{eq55f}
 \mathcal{C}_{2m+2}(u,z)
 =
 u^{4/3} (\zeta f(z))^{1/2}
 \mathcal{B}_{2m+2}(u,z) +
 \partial \mathcal{A}_{2m+2}(u,z) / \partial z
\end{equation}
\begin{equation}
\label{eq55g}
\mathcal{D}_{2m+2}(u,z)
=
u^{2/3} (f(z)/\zeta)^{1/2}
\mathcal{A}_{2m+2}(u,z) 
+ \partial \mathcal{B}_{2m+2}(u,z) / \partial z
\end{equation}
and \footnote{In \cite{Dunster:2020:ASI} $u^{-2/3}$ should read $u^{2/3}$}
\begin{equation}
\label{eq55h}
 \mathcal{H}_{m}(u,z)
 =\partial \mathcal{G}_{m}(u,z)/ \partial z
 +u^{2/3} \gamma_{m}(u) (f(z)/\zeta)^{1/2}
 \mathcal{B}_{2m+2}(u,z)
\end{equation}
\end{theorem}

\begin{proof}
We begin by defining $\mathcal{G}_{m}(u,z)$ by (\ref{eq55}) for the values $(j,k)=(-1,1)$, so that
\begin{multline}
\label{eq55a}
\mathcal{G}_{m}(u,z):= w^{(-1,1)}(u,z) \\
-\gamma_{m}(u)\left\{\mathrm{Wi}^{(-1,1)}\left(u^{2/3}\zeta  \right)\mathcal{A}_{2m+2}(u,z) + {\mathrm{Wi}^{(-1,1)}}^{\prime}\left(u^{2/3}\zeta \right)\mathcal{B}_{2m+2}(u,z) \right\}.
\end{multline}
This of course establishes the stated analyticity in $Z$.

Now from (\ref{eq23}) and (\ref{eq55a})
\begin{multline}
\label{eq56}
 w^{(-1,1)}(u,z) 
-2\pi e^{-\pi i/6}\gamma_{m}(u)w_{m,1}(u,z) \\
 =\gamma_{m}(u)\left\{\mathrm{Wi}^{(0,1)}\left( 
u^{2/3}\zeta \right)\mathcal{A}_{2m+2}(u,z)
+{\mathrm{Wi}^{(0,1)}}' \left(u^{2/3}\zeta \right)\mathcal{B}_{2m+2}(u,z) \right\} +\mathcal{G}_{m}(u,z).
\end{multline}
Then on appealing to (\ref{eq54}) we see (\ref{eq55}) holds for $(j,k)=(0,1)$. Moreover, since (\ref{eq55}) holds for both $(j,k)=(-1,1)$ and $(j,k)=(0,1)$ it follows that $\mathcal{G}_{m}(u,z)$ is bounded in $Z^{(-1,1)}(u)\cup Z^{(0,1)}(u)$, since all the other functions in (\ref{eq55}) are bounded in $Z^{(j,k)}(u)$. In particular $\mathcal{G}_{m}(u,z)$ is bounded as $z \rightarrow z^{(j)}$ for $j=0,\pm1$.

Next from (\ref{eq24}) and (\ref{eq55a}) we have
\begin{multline}
\label{eq58}
 w^{(-1,1)}(u,z)-2\pi e^{\pi i/6}\gamma_{m}(u)w_{m,-1}(u,z) \\ 
 =\gamma_{m}(u)\left\{\mathrm{Wi}^{(-1,0)}\left(u^{2/3}\zeta 
\right)\mathcal{A}_{2m+2}(u,z)
+{\mathrm{Wi}^{(-1,0)}}'\left(
u^{2/3}\zeta \right)\mathcal{B}_{2m+2}(u,z) \right\} \\ +\mathcal{G}_{m}(u,z).
\end{multline}
Now the LHS is a solution of the inhomogeneous equation (\ref{eq1}), and hence of course so too is the RHS. But all the functions appearing on the RHS are also seen to be bounded as $z \rightarrow z^{(j)}$ for both $j=0$ and $j=-1$, and hence by uniqueness it must be equal to $w^{(-1,0)}(u,z)$. This establishes (\ref{eq55}) for the final value $(j,k)=(-1,0)$, and consequently $\mathcal{G}_{m}(u,z)$ is also bounded in $Z^{(-1,0)}(u)$.

Finally (\ref{eq55e}) - (\ref{eq55h}) come from differentiating (\ref{eq55}) with respect to $z$, using $d\zeta/ dz=\{f(z)/\zeta\}^{1/2}$ (see (\ref{xi})), and invoking (\ref{eq27}).
\end{proof}

\begin{theorem}
\label{thmconnection}
Connection formulas are given by (\ref{eq54}) along with
\begin{equation}
\label{eq60}
w^{(-1,0)}(u,z)=w^{(-1,1)}(u,z)-2\pi e^{\pi i/6}\gamma_{m}(u)w_{m,-1}(u,z),
\end{equation}
and
\begin{equation}
\label{eq61}
w^{(-1,0)}(u,z)=w^{(0,1)}(u,z)-2\pi i\gamma_{m}(u)w_{m,0}(u,z).
\end{equation}
\end{theorem}

\begin{proof}
From (\ref{eq55}) for $(j,k)=(-1,0)$ and (\ref{eq58}) we get (\ref{eq60}). From (\ref{eq52}), (\ref{eq54}) and (\ref{eq60}) we get (\ref{eq61}).
\end{proof}
\begin{remark}
These connection formulas allow the extension of the asymptotic approximations of \cref{away} to regions where the solutions are exponentially large. Moreover they can be used to improve accuracy in regions where the slowly-varying expansions begin to break down, as we will discuss in \cref{secpoly}.
\end{remark}

\subsection{Computation of $\gamma_{m}(u)$ and $\mathcal{G}_{m}(u,z)$}
We next address the issue of computation of the approximations given by (\ref{eq55}). The coefficient functions $\mathcal{A}_{2m+2}(u,z)$ and $\mathcal{B}_{2m+2}(u,z)$ are computable by (\ref{eq48}) - (\ref{eq51}), and we shall use similar methods to compute our new coefficient function $\mathcal{G}_m(u,z)$. 

Before showing how to do this, consider the the connection coefficient $\gamma_{m}(u)$; if isn't known explicitly then we can approximate it as follows. Let $\Gamma$ be a simple positively orientated loop which encloses the turning point $z_{0}$ and which lies in $Z^{(\mathcal{G})}(u)$. Then by the Cauchy-Goursat theorem
\begin{equation}
\label{eq62}
0=\oint_{\Gamma} {\mathcal{G}_{m}(u,z)dz}
=\sum\limits_{j,k} \int_{\Gamma^{(j,k)}(u)}
{\mathcal{G}_{m}(u,z)dz},
\end{equation}
where we have broken $\Gamma$ into the union of three paths, with $0\leq \arg(u^{2/3}\zeta) \leq 2 \pi/3$ for  $\Gamma^{(0,1)}(u)$, $-2 \pi/3 \leq \arg(u^{2/3}\zeta) \leq 0$ for  $\Gamma^{(-1,0)}(u)$, and $|\arg(-u^{2/3}\zeta)| \leq \pi/3$ for  $\Gamma^{(-1,1)}(u)$. Note $\Gamma^{(j,k)}(u) \in Z^{(j,k)}(u)$ for all three.

Next we solve (\ref{eq55}) for $\mathcal{G}_{m}(u,z)$ to give three expressions for it, and plug these into the three integrals in the sum of (\ref{eq62}) for the corresponding values of $(j,k)$. As a result we get
\begin{multline}
\label{eq63}
 \gamma_{m}(u)\sum\limits_{j,k} \int_{\Gamma^{(j,k)}(u)} \left\{\mathrm{Wi}^{(j,k)}\left(u^{2/3}\zeta \right)\mathcal{A}_{2m+2}(u,z) \right. \\
 \left. +{\mathrm{Wi}^{(j,k)}}'\left(u^{2/3}\zeta \right)\mathcal{B}_{2m+2}(u,z) \right\}dz =\sum\limits_{j,k} \int_{\Gamma^{(j,k)}(u)} w^{(j,k)}(u,z)dz.
\end{multline}

The key now is that each of the three functions $\mathrm{Wi}^{(j,k)}(u^{2/3}\zeta)$ has the same asymptotic expansion as the other two on its path $\Gamma^{(j,k)}(u)$, with the only differences being the error terms. The same is true for each of $w^{(j,k)}(u,z)$, and this is the reason why we split the integral in (\ref{eq62}) as we did.

Consequently, from (\ref{eq48}), (\ref{eq49}), (\ref{eq42}), (\ref{eq28}), and (\ref{eq29}), and recombining the three paths into their parent $\Gamma$, we arrrive at our desired expansion
\begin{equation}
\label{eq64}
u^{4/3}\gamma_{m}(u)=
\sum\limits_{s=0}^m 
\oint_{\Gamma}\frac{\hat{G}_{s}(z)dz}{u^{2s}}
\left[\oint_{\Gamma} {\frac{1}{\zeta }\left\{\frac{\zeta }{f(z)} \right\}^{1/4}J_{m}(u,z)dz}\right]^{-1}
+\mathcal{O}\left( \frac{1}{u^{2m+2}} \right),
\end{equation}
where
\begin{multline}
\label{eq65}
 J_{m}(u,z)=-\exp \left\{\sum\limits_{s=1}^{m} \frac{\tilde{\mathcal{E}}_{2s}(z)}{u^{2s}} \right\}\cosh \left\{\sum\limits_{s=0}^{m} 
\frac{\tilde{\mathcal{E}}_{2s+1}(z)}{u^{2s+1}}  
\right\}\sum\limits_{k=0}^{m} \frac{(3k)!}{k!\left( 3u^{2}\zeta^{3} \right)^{k}} \\ 
 +\frac{1}{u\zeta^{3/2}}\exp \left\{\sum\limits_{s=1}^m \frac{{\mathcal{E}}_{2s}(z)}{u^{2s}}  \right\}\sinh \left\{\sum\limits_{s=0}^{m} 
\frac{{\mathcal{E}}_{2s+1}(z)}{u^{2s+1}}  \right\}\sum\limits_{k=0}^{m}
\frac{(3k+1)!}{k!\left(3u^{2}\zeta^{3} \right)^{k}}.
\end{multline}

In (\ref{eq64}) an error bound for the $\mathcal{O}$ term can be constructed using those associated with the referenced expansions used in its construction. This involves taking the supremum of these bounds over the paths $\Gamma^{(j,k)}(u)$. Although the derivation is straightforward the result is somewhat unwieldy, so we omit details.

The integrals involving $\hat{G}_{s}(z)$ in (\ref{eq64}) can be readily computed via the trapezoidal rule or other numerical methods for contour integrals. The same is true for the one involving $J_m(u,z)$. However for the former it is simpler than that, since each $\hat{G}_{s}(z)$ generally has a pole at the turning point $z=z_{0}$ (of order $3s+1$ if $p(z_{0})\neq 0$). Hence we can use the exact expression
\begin{equation}
\label{eq65c}
    \oint_{\Gamma} {\hat{G}_{s}(z)dz} =2\pi i \underset{z=z_{0}}{\operatorname{Res}}\left\{  \hat{G}_{s}(z) \right\}.
\end{equation}
For example, from (\ref{eq43})
\begin{equation}
\label{eq65d}
    \oint_{\Gamma} {\hat{G}_{0}(z)dz} =
    -\frac{2\pi i p(z_{0})}{{f}'(z_{0})}.
\end{equation}

Next consider the computation of $\mathcal{G}_{m}(u,z)$. Of course if $z$ is not close to the turning point we do not need to evaluate it, since in this case the fundamental solutions $w^{(j,k)}(u,z)$ can be approximated by \cref{thm6}. That being said, it does have the following expansion which is uniformly valid for $z \in Z^{(\mathcal{G})}(u)$ and $z$ bounded away from $z_0$
\begin{equation}
\label{eq66}
\mathcal{G}_{m}(u,z)=\frac{1}{u^{2}}\sum\limits_{s=0}^m \frac{\hat{G}_{s} 
(z)}{u^{2s}} -\frac{\gamma_{m}(u)J_{m}(u,z)}{u^{2/3}\zeta }\left\{\frac{\zeta }{f(z)}\right\}^{1/4}+\mathcal{O}\left(\frac{1}{u^{2m+4}} \right).
\end{equation}
This follows from (\ref{eq48}), (\ref{eq49}), (\ref{eq42}), (\ref{eq28}), (\ref{eq29}) and (\ref{eq55}).

We can then use this in the Cauchy integral formula
\begin{equation}
\label{eq67}
\mathcal{G}_{m}(u,z)=\frac{1}{2\pi i}\oint_{\Gamma } 
\frac{\mathcal{G}_{m}(u,t)dt}{t-z},
\end{equation}
where $\Gamma$ is as above and $z$ lies in its interior. As a result we arrive at our desired expansion
\begin{multline}
\label{eq68}
\mathcal{G}_{m}(u,z)=\frac{1}{2\pi iu^{2}}\sum\limits_{s=0}^m
\frac{1}{u^{2s}}\oint_{\Gamma} 
{\frac{\hat{G}_{s}(t)dt}{t-z}}  \\ 
 -\frac{\gamma_{m}(u)}{2\pi iu^{2/3}}\oint_{\Gamma} {\left\{ {\frac{\zeta(t)}{f(t)}} \right\}^{1/4}\frac{J_{m}(u,t)dt}{\zeta(t)(t-z)}} +\mathcal{O}\left( \frac{1}{u^{2m+4}} \right),
\end{multline}
which can be used for $z$ in a neighbourhood of the turning point. Again an error bound for the $\mathcal{O}$ term can be constructed in a manner similar to the one obtainable for (\ref{eq64}).

The integrals can be computed similarly to those in (\ref{eq64}). In particular, as  observed above, each $\hat{G}_{s}(z)$ generally has a pole at the turning point. Therefore the following generalisation of Cauchy's integral formula can aid in the computation of the loop integrals involving $\hat{G}_{s}(t)$ in (\ref{eq68}).
\begin{theorem} 

\label{thm:cauchy} 

Let $C$ be a positively orientated simple loop in the $z$ plane, and $G(z)$ be a function that is analytic in the open region enclosed by the path and continuous on its closure, except for a pole of arbitrary order $p$ at an interior point $z=z_{0}$. Let $\left\{ g_{j}\right\}_{j=-\infty }^{\infty }$ be the Laurent coefficients of $G(z)$ at $z=z_{0}$, so that for some $r_{0}>0$ 
\begin{equation}
G(z) =\sum_{j=-p}^{\infty }g_{j}\left( z-z_{0}\right) ^{j}\
\left( 0<\left\vert z-z_{0}\right\vert <r_{0}\right) ,  \label{2.9}
\end{equation}
and let $G^{\ast}(z)$ denote the regular (or analytic) part of $G(z)$ at $z=z_{0}$, given by 
\begin{equation}
G^{\ast}(z) =\sum_{j=0}^{\infty }g_{j}\left( z-z_{0}\right)
^{j}\ \left( 0\leq \left\vert z-z_{0}\right\vert <r_{0}\right) .
\label{2.10}
\end{equation}
Then for all $z$ lying inside $C$ 
\begin{equation}
\oint_{C}\frac{G(t)}{t-z}dt{=2\pi iG}^{\ast}(z).  
\label{2.11}
\end{equation}

\end{theorem}

\begin{proof} 

We split the LHS of (\ref{2.11}) in the form 
\begin{equation}
\oint_{C}\frac{G(t)}{t-z}dt{=}I_{1}(z)+I_{2}(z),  
\label{2.12}
\end{equation}
where 
\begin{equation}
I_{1}(z)=\oint_{C}\frac{G(t)-G^{\ast}(t)}{t-z}dt,  
\label{2.13}
\end{equation}
and 
\begin{equation}
I_{2}(z)=\oint_{C}\frac{G^{\ast}(t)}{t-z}dt.  
\label{2.13a}
\end{equation}

Now from (\ref{2.9}) and (\ref{2.10}) 
\begin{equation}
G(z)-G^{\ast}(z)=\sum_{j=1}^{p}\frac{g_{-j}}{\left( z-z_{0}\right) ^{j}},
\label{2.14}
\end{equation}
which is analytic for $0<\left\vert z-z_{0}\right\vert <\infty $. Therefore
we can deform the contour in $I_{1}(z)$ to the circle $\left\vert t-z_{0}\right\vert =R>0$ which contains $t=z$, but where $R$ can otherwise
be arbitrarily chosen. Thus from (\ref{2.13}) and (\ref{2.14}) 
\begin{equation}
I_{1}(z)=\sum_{j=1}^{p}\oint_{\left\vert t-z_{0}\right\vert =R}\frac{g_{-j}dt
}{\left( t-z_{0}\right) ^{j}(t-z) }.  
\label{2.15}
\end{equation}

Next assuming $R>|z-z_{0}| $ we have for each $t$ lying on the circle $|t-z_{0}| =R$ that $|t-z| =|(t-z_{0})-(z_{0}-z)| \geq \left\vert
|t-z_{0}| - |z_{0}-z| \right\vert
=R-|z-z_{0}| $ ($>0$). Hence with the triangle
inequality we deduce that 
\begin{equation}
\left\vert I_{1}(z)\right\vert \leq \frac{2\pi pM}{R-|z-z_{0}| }=\mathcal{O}\left( \frac{1}{R}\right) \quad (R\rightarrow \infty ),
\label{2.16}
\end{equation}
where $M=\max_{1\leq j\leq p}|g_{-j}|$. Since $I_{1}(z)$ is independent of $R
$ we conclude it is identically zero, and hence from (\ref{2.12}) 
\begin{equation}
\oint_{C}\frac{G(t)}{t-z}dt=I_{2}(z).  
\label{2.17}
\end{equation}

Finally $G^{\ast}(z)$ is analytic at every point in the open region enclosed by $C$, and continuous on its closure. Hence from (\ref{2.13a})\ and Cauchy's integral formula we conclude that $I_{2}(z)=2\pi iG^{\ast}(z)$, and hence (\ref{2.11}) follows from (\ref{2.17}). 
\end{proof}

It follows from \cref{thm:cauchy} that in (\ref{eq68}) we can use
\begin{equation}
\label{eq68a}
\frac{1}{2 \pi i}\oint_{\Gamma} 
{\frac{\hat{G}_{s}(t)dt}{t-z}} =
\hat{G}_{s}^\ast(z),
\end{equation}
where $\hat{G}_{s}^\ast(z)$ is the regular (analytic) part of $\hat{G}_{s}(z)$ at the turning point. For example,
\begin{equation}
\hat{G}_{0}^\ast(z)=\frac{p\left(z_{0}\right)}
{f'\left(z_{0}\right)\left(z-z_{0}\right) }-\frac{p(z)}{f(z)}.
\end{equation}

The function $J_m(u,z)$ typically has an essential singularity at the turning point, and \cref{thm:cauchy} can be generalised to accommodate this. However, it is generally difficult to compute the required regular part of $J_m(u,z)$. Nonetheless, the contour integrals involving this function can be directly computed to high accuracy using the trapezoidal method since it is slowly varying and analytic on the closed contour (see \cite{Bornemann:2011:AAS} for a general discussion).

For the derivatives of the solutions in \cref{ThmG} one can modify the integrands of (\ref{eq50}) and (\ref{eq51}) appropriately for the terms involving $\mathcal{A}_{2m+2}(u,z)$ and $\mathcal{B}_{2m+2}(u,z)$ appearing in (\ref{eq55f}) - (\ref{eq55h}). For computing the derivatives of these functions and $\mathcal{G}_{m}(u,z)$ use (\ref{eq50}), (\ref{eq51}) and (\ref{eq68}) with $t-z$ in each denominator replaced by $(t-z)^2$. In doing so for (\ref{eq68}) we are able to utilise the differentiated version of (\ref{eq68a}), namely
\begin{equation}
\label{eq68b}
\frac{1}{2 \pi i}\oint_{\Gamma} 
{\frac{\hat{G}_{s}(t)dt}{(t-z)^2}} =
\hat{G}_{s}^{\ast \prime}(z).
\end{equation}

\section{Inhomogeneous Airy equation}
\label{sec6}
We apply the foregoing approximations to the inhomogenous Airy equation
\begin{equation}
\label{eq69}
\frac{d^{2}w}{dz^{2}}-u^{2}zw=p(z),
\end{equation}
where $u>0$. Evidently the Liouville transformation (\ref{xi}) is not necessary here, so throughout this section $z$ and $\zeta$ are equivalent.

Using the Wronskians (\ref{eq70}) and (\ref{eq71}), and the asymptotic behavior (\ref{Aiinfinity}) of the Airy function in the complex plane, we obtain by variation of parameters these exact representations of the three fundamental solutions
\begin{multline}
\label{eq72}
 w^{(-1,1)}(u,z)=2\pi i u^{-2/3} \left\{
 \mathrm{Ai}_{-1} \left(u^{2/3}z 
\right)\int_{\infty \exp (2\pi i/3)}^{z} {p(t)\mathrm{Ai}_{1} \left(u^{2/3}t 
\right)dt} \right. \\ 
 \left. -\mathrm{Ai}_{1} \left(u^{2/3}z \right)\int_{\infty \exp 
(-2\pi i/3)}^{z} {p(t)\mathrm{Ai}_{-1} \left(u^{2/3}t \right)dt} \right\},
\end{multline}
\begin{multline}
\label{eq73}
 w^{(-1,0)}(u,z)=2\pi e^{\pi i/6}u^{-2/3} \left\{
 \mathrm{Ai}_{-1} \left(u^{2/3}z 
\right)\int_{\infty}^{z} {p(t)\mathrm{Ai} \left(u^{2/3}t \right)dt} \right. \\ \left.
 -\mathrm{Ai} \left(u^{2/3}z 
\right)\int_{\infty \exp (-2\pi i/3)}^{z} {p(t)\mathrm{Ai}_{-1} \left( 
{u^{2/3}t} \right)dt}  \right\},
\end{multline}
and
\begin{multline}
\label{eq74}
 w^{(0,1)}(u,z)=2\pi e^{-\pi i/6}u^{-2/3}
 \left\{ 
 \mathrm{Ai}_{1} \left(u^{2/3}z 
\right)\int_{\infty}^{z} {p(t)\mathrm{Ai} \left(u^{2/3}t \right)dt} \right. \\  \left.
 -\mathrm{Ai} \left(u^{2/3}z
\right)\int_{\infty \exp (2\pi i/3)}^{z} {p(t)\mathrm{Ai}_{1} \left(u^{2/3}t 
\right)dt}
\right\}.
\end{multline}

For the case $p(z)$ equal to a constant these of course reduce to a constant multiple of the Scorer functions (\ref{eq22a}) - (\ref{eq22c}) of argument $u^{2/3}z$; see also (\ref{eq4}) and (\ref{eq28}).

Asymptotic approximations for various classes of forcing functions are obtainable from \cref{thm6} ($z$ bounded away from the turning point at 0) and \cref{ThmG} ($z$ in a domain containing 0). We shall consider two important cases, $p(z)$ being a polynomial or an exponential function.

\subsection{Polynomial forcing term}
\label{secpoly}
Here we consider $p(z)$ a general (non-constant) polynomial, so that for some integer $R\geq 1$
\begin{equation}
\label{eq75}
p(z)=\sum\limits_{r=0}^R {p_{r} z^{r}} \quad (p_{R}\neq0).
\end{equation}
We shall use:
\begin{lemma}
For any solution $y(z)$ of Airy's equation we have
\begin{equation}
\label{eq78}
\int zy(z)dz=y'(z),
\end{equation}
\begin{equation}
\label{eq79}
\int z^{2}y(z)dz=z y'(z)-y(z),
\end{equation}
and for $r=3,4,5,\cdots $
\begin{equation}
\label{eq81}
\int {z^{r}y(z)dz} =P_{r-2}(z)y(z)+Q_{r-1}(z)y'(z)+c_{r} \int {y(z)dz}.
\end{equation}
Here 
\begin{equation}
\label{eqcr}
  c_{r} =
  \begin{cases}
     (3k)!/\left( {3^{k}k!} \right) & (r=3k, \, k=1,2,3,\cdots) \\
     0 & (\emph{otherwise})
  \end{cases},
\end{equation}
and $P_{r}(z)$ and $Q_{r}(z)$ are polynomials of degree $r$ given by
\begin{equation}
\label{eq82}
P_{r}(z)=-Q'_{r+1}(z),
\end{equation}
and
\begin{equation}
\label{eq83}
Q_{r}(z)=\sum\limits_{j=0}^{\left\lfloor {r/3} \right\rfloor } {q_{r-3j} 
(r)z^{r-3j}} \quad (r=3,4,5,\cdots ),
\end{equation}
where $q_{r} (r)=1$, and working backwards
\begin{equation}
\label{eq84}
q_{j} (r)=(j+2)(j+3)q_{j+3} (r)\quad (j=r-3,r-6,\cdots ,j_{0} ),
\end{equation}
in which $j_{0} =r-3\left\lfloor {r/3} \right\rfloor =r\bmod 3$. 
\end{lemma}

\begin{remark}
This lemma can be used to give a general representation of the iterated integration of the Airy functions, but we do not give details since they are not applicable here. We mention that in \cite{Abramochkin:2018:HDO} similar expressions were derived for the higher derivatives of Airy functions in a closed form, which also involve the Airy function, its first derivative and certain polynomials.
\end{remark}

\begin{proof}
From \cite[§9.10(iii)]{NIST:DLMF} we have (\ref{eq78}), (\ref{eq79}) and for $n=0,1,2,\cdots$
\begin{equation}
\label{eq80}
\int z^{n+3}y(z)dz=-(n+2)z^{n+1}y(z)+z^{n+2}y'(z)+(n+1)(n+2)\int z^{n}y(z)dz.
\end{equation}
The form (\ref{eq81}) for polynomial $P_{r-2}(z)$ and $Q_{r-1}(z)$ of degree $r-2$ and $r-1$ respectively, and the value (\ref{eqcr}) of $c_{r}$, can be inferred from by induction from (\ref{eq78}), (\ref{eq79}) and (\ref{eq80}).

Next on differentiating (\ref{eq81}) we obtain
\begin{equation}
\label{84a}
    \left\{ P_{r-2}(z)  +Q_{r-1}'(z)
 \right\} y'(z)
 + \left\{ P_{r-2}'(z) + z Q_{r-1}(z) +c_{r}-z^{r} \right\} y(z)=0.
\end{equation}
Setting the coefficient of $y'(z)$ to zero gives (\ref{eq82}). Plugging this into the coefficient of $y(z)$ and setting this to zero yields
\begin{equation}
\label{84b}
    Q_{r-1}''(z)
  - z Q_{r-1}(z) -c_{r}+z^{r} =0.
\end{equation}
We next write 
\begin{equation}
\label{eq84c}
Q_{r-1}(z)=\sum\limits_{j=0}^{r-1} {q_{j}(r)z^{j}}
\end{equation}
insert this into (\ref{84b}), and collect like powers of $z$. As a result we find (on suppressing $r$ dependence on the $q_{j}$)
\begin{multline}
     \left( q_{r-1}-1 \right){x}^{r}
     +q_{r-1}x^{r-1}+q_{r-2}x^{r-2}
    \\
    +\sum _{j=1}^{r-3} 
    \left( q_{j-1}-(j+2)(j+1)q_{j+2} \right) {x}^{j}
    +c_{r}-2q_{2}=0.
\end{multline}
On equating the coefficients of each power of $z$ to zero, and replacing $r-1$ by $r$, yields (\ref{eq83}) and (\ref{eq84}).
\end{proof}

In (\ref{eq83}) the term of smallest power is $q_{j_{0}} (r) z^{j_{0}}$, and from its definition $j_{0} =0,1$ or 2. Furthermore, from an inductive argument it can readily be verified from (\ref{eq84}) that
\begin{equation}
\label{eq85}
q_{j} (r)=\frac{r!}{(j+1)!}\left\{ {\prod\limits_{k=0}^{\left\lfloor 
{(r-j)/3-2} \right\rfloor } {(r-2-3k)} } \right\}^{-1},
\end{equation}
where the product is unity if the upper limit is negative.

The first few $Q$ polynomials are given by
\begin{multline}
\label{eq86}
 Q_{1}(z)=z,\;Q_{2}(z)=z^{2},\;Q_{3}(z)=z^{3}+6,\;  
 Q_{4}(z)=z^{4}+12z, \\
 Q_{5}(z)=z^{5}+20z^{2},\;Q_{6}(z)=z^{6}+30z^{3}+180. 
\end{multline}

Let us now apply \cref{gz} to (\ref{eq75}), and on identifying this equation with (\ref{eq1}) we see that $f(z)=z$, $g(z)=0$. Hence from (\ref{eq5}) we find that $\Phi(z)=-\frac{5}{16}z^{-3}$. Since $p(z)$ is entire and $\Phi(z)$ is analytic for $z\neq0$, and noting that $z=\zeta$, the asymptotic solution $w^{(j,k)}(u,z)$ given by \cref{gz} is valid for $z \in \mathbf{S}_{j} \cup \mathbf{S}_{k}$, except near the boundaries and near $z=0$. More precisely, for arbitrary positive small $\delta$ the expansion for $w^{(j,k)}(u,z)$ is uniformly valid for $z$ lying in the domain $Z_{\delta}^{(j,k)}$, where
\begin{equation}
\label{eq86a}
Z_{\delta}^{(0,1)} = \left\{ 
z:-\tfrac{1}{3}\pi + \delta \leq \arg(z) \leq \tfrac{2}{3} \pi-\delta,
\, |z|\geq \delta
\right\},
\end{equation}
with $Z_{\delta}^{(-1,0)}$ being the conjugate of this, and $Z_{\delta}^{(-1,1)}$ is $Z_{\delta}^{(-1,0)}$ rotated positively by $2 \pi/3$.
We have then
\begin{theorem}
For $z \in Z_{\delta}^{(j,k)}$ asymptotic solutions of (\ref{eq70}) with $p(z)$ given by (\ref{eq75}) are furnished by (\ref{eq42}), where
\begin{equation}
\label{eq95}
\hat{G}_{0}(z)=-\frac{p(z)}{f(z)}=-\sum\limits_{r=0}^R {p_{r} z^{r-1}},
\end{equation}
\begin{equation}
\label{eq96}
\hat{G}_{s+1}(z)=z^{-1}\hat{G}''_{s}(z) \quad (s=0,1,2,\cdots),
\end{equation}
and with the error bound
\begin{multline}
\left\vert \hat{\varepsilon}_{n}^{(j,k)}(u,z)\right\vert \leq \frac{1}{
u^{2n+2}}\left\{\left\vert {\hat{G}_{n}(z)}\right\vert +\frac{1}{2|z|^{1/4}
}\int_{\hat{\mathcal{L}}^{(j,k)}(z)}\left\vert \left\{ t^{1/4}\hat{G}
_{n}(t)\right\} ^{\prime }dt\right\vert \right\}  \\
+\dfrac{5\hat{L}_{n}^{(j,k)}(u,z)}{32u^{2n+3}|z|^{1/4}}\left\{ {1-\dfrac{5
}{32u}\int_{\hat{\mathcal{L}}^{(j,k)}(z)}{\left\vert \frac{dt}{t^{5/2}}
\right\vert }}\right\} ^{-1}\int_{\hat{\mathcal{L}}^{(j,k)}(z)}{{
\left\vert \frac{dt}{t^{5/2}}\right\vert }}, \label{eq96a}
\end{multline}
where 
\begin{equation}
\hat{L}_{n}^{(j,k)}(u,z)=\sup_{t\in \hat{\mathcal{L}}^{(j,k)}(z)}\left
\vert t^{1/4}{\hat{G}}_{n}(t)\right\vert +\frac{1}{2}\int_{\hat{\mathcal{L}}
^{(j,k)}(z)}{\left\vert {\left\{ t^{1/4}{\hat{G}}_{n}(t)\right\} ^{\prime
}dt}\right\vert }. 
\label{eq96b}
\end{equation}
\end{theorem}

\begin{remark}
For the integrals involving $\{ t^{1/4}{\hat{G}}_{n}(t)\}'$ to converge at infinity we require that $\hat{G}_{n}(z)=\mathcal{O}(z^{-1})$ as $z\rightarrow \infty $, so that from (\ref{eq95})
and (\ref{eq96}) $n\geq \frac{1}{3}R$. If fewer terms are taken the bound can be modified as in (\ref{LG87a}). In either case $\hat{\varepsilon}_{n}^{(j,k)}(u,z) =\mathcal{O}(u^{-2n-2})$ as $u \rightarrow \infty$.
\end{remark}

\begin{figure}[htbp]
  \centering
  \includegraphics[width=\textwidth,keepaspectratio]{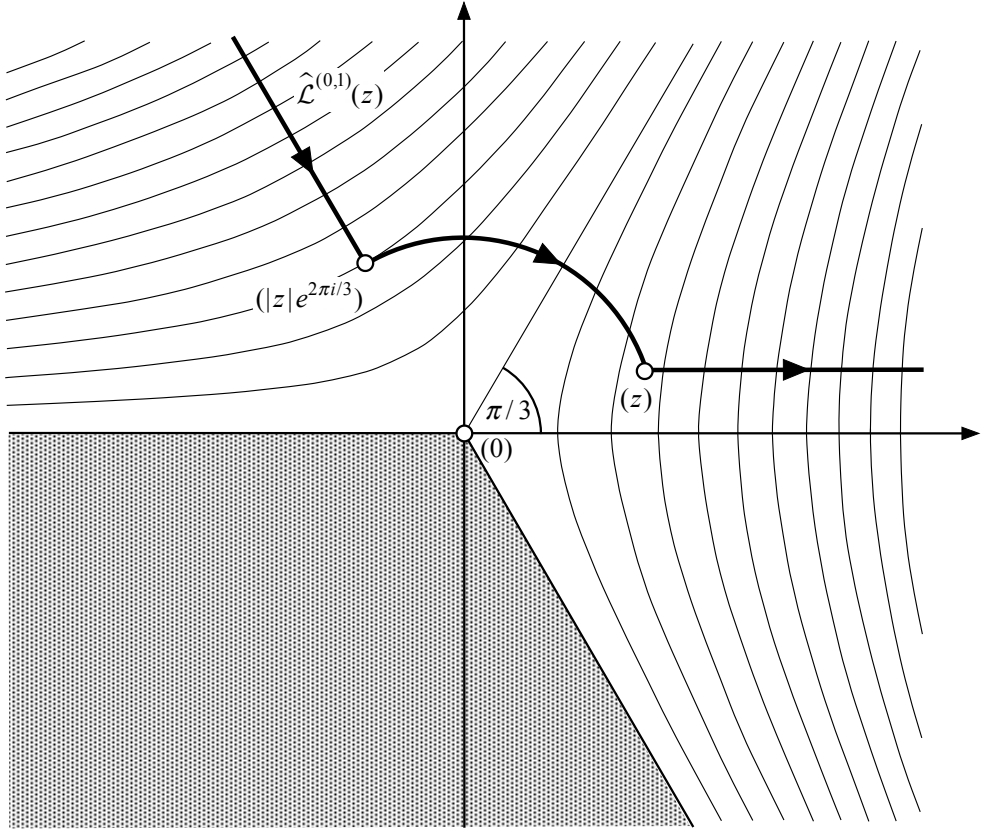}
  \caption{Path $\hat{\mathcal{L}}^{(0,1)}(z)$ in $t$ plane for $u >0$ and $0\leq\arg(z)\leq 2\pi/3$}
  \label{fig:fig1}
\end{figure}

\begin{figure}[htbp]
  \centering
  \includegraphics[width=\textwidth,keepaspectratio]{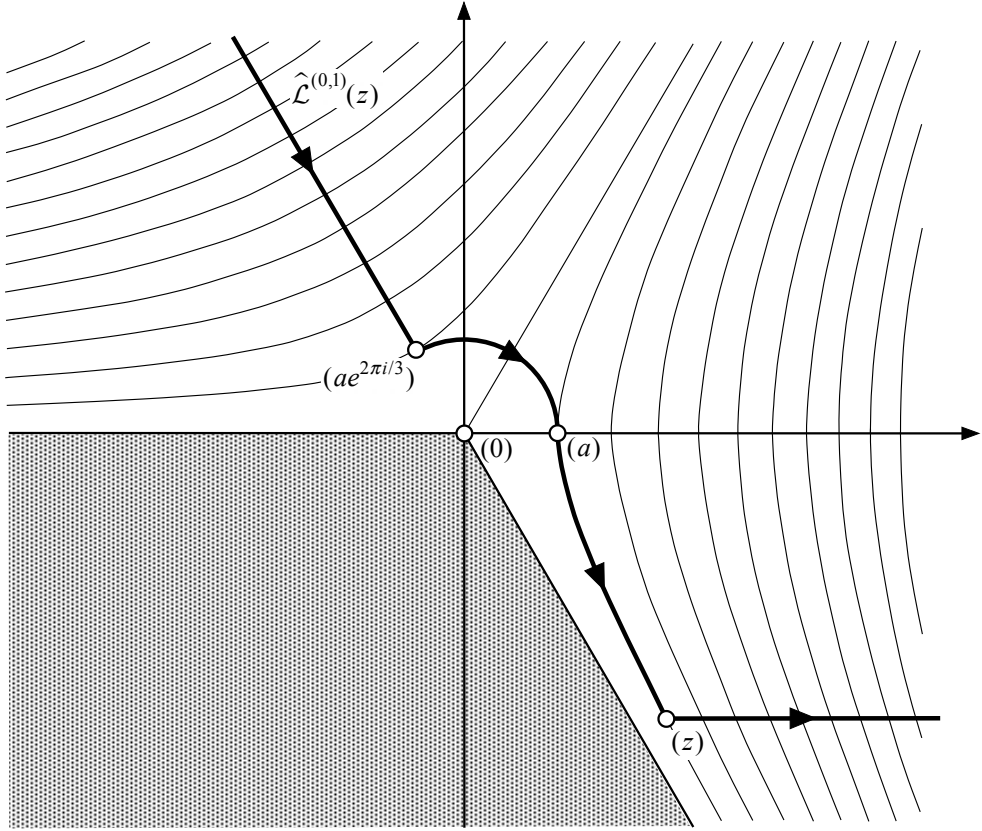}
  \caption{Path $\hat{\mathcal{L}}^{(0,1)}(z)$ in $t$ plane for $u >0$, $-\pi/3 < \arg(z) < 0$ and $a=\{\Re(z^{3/2})\}^{2/3}$}
  \label{fig:fig2}
\end{figure}

Recalling $u >0$, in \cref{fig:fig1} we illustrate  several curves in $t$ plane given by
\begin{equation}
\Re(t^{3/2})=\text{constant},
\end{equation}
and on these we observe $\Re(t^{3/2}) \rightarrow -\infty$ as $t \rightarrow \infty e^{2 \pi i/3}$ and $\Re(t^{3/2}) \rightarrow +\infty$ as $\Re(t) \rightarrow +\infty$.

Also included is a typical path $\hat{\mathcal{L}}^{(0,1)}(z)$ for $0\leq\arg(z)\leq 2\pi/3$. It consists of the union of the ray from $t=\infty  e^{2 \pi i/3}$ to $t = |z|e^{2 \pi i/3}$, the negatively orientated arc $|t| = |z|$ from $t = |z| e^{2 \pi i/3} $ to $t=z$, and the horizontal line from $t=z$ to $t=i\Im(z)+\infty$.

This path connects $z^{(1)}$ to $z^{(0)}$, passes through $z$, and meets the monotonicity requirement that $\Re(z^{3/2})$ be nondecreasing as $t$ passes along it from $z^{(1)}$ to $z^{(0)}$. Moreover, the inclusion of the arc segment ensures that $|t|\geq |z|$ for all $t$ on the path, and hence the bounds (\ref{eq96a}) and (\ref{eq96b}) are of the appropriate order of magnitude as $z\rightarrow \infty$. We also note that the path keeps the maximum possible distance away from the singularity at $t=0$.

In \cref{fig:fig2} a typical path $\hat{\mathcal{L}}^{(0,1)}(z)$ in $t$ plane is shown for $-\pi/3 < \arg(z) < 0$. In this $a=\{\Re(z^{3/2})\}^{2/3}>0$. The path consists of the union of the ray from $t=\infty  e^{2 \pi i/3}$ to $t = a e^{2 \pi i/3}$, the negatively orientated arc $|t| = a$ from $t = a e^{2 \pi i/3} $ to $t=a$, the part of the level curve $\Re(t^{3/2})=\Re(z^{3/2})$ from $t=a$ to $t=z$, and the horizontal line from $t=z$ to $t=i\Im(z)+\infty$.

This path meets all the requirements for the error bounds to be valid, and in particular in order not to violate the monotonicity condition we are required to include a segment of the level curve passing through $t=z$. Consequently the path must pass within a distance $a$ of the pole at $t=0$. We note that $a\rightarrow 0$ as $z$ approaches the ray $\arg(z)=-\pi /3$, which indicates that the approximation breaks down at this boundary, as expected. Furthermore, $a$ remains bounded as $z \rightarrow \infty$ along a fixed level curve, and so the bounds lose sharpness in this case too.

All of the above problems, along with similar ones when $2\pi/3<\arg(z)\leq \pi$, can be avoided by using the connection formulas of \cref{thmconnection}. For example, if $-\pi/3 \leq \arg(z) < 0$ we can use (\ref{eq61a}). To do so we observe, from (\ref{eq47}) with $z=\zeta$, $f(z)=z$ and $g(z)=0$, that $\mathcal{A}_{2m+2}(u,z)=1$ and $\mathcal{B}_{2m+2}(u,z)=0$. Thus we have from (\ref{eq61})
\begin{equation}
\label{eq61a}
w^{(0,1)}(u,z)=
w^{(-1,0)}(u,z)+2\pi i\gamma(u)
\mathrm{Ai} \left(u^{2/3}z\right),
\end{equation}
where $\gamma(u)$ is independent of $m$ (and is given explicitly by (\ref{eq101m}) below). In this the asymptotic expansion for $w^{(-1,0)}(u,z)$ is again given by (\ref{eq42}), but with the new path $\hat{\mathcal{L}}^{(-1,0)}(z)$ in the error bound (\ref{eq96a}) extending from $z^{(-1)}=\infty e^{-2 \pi i/3} $ to $z^{(0)} =\Im(z) +\infty$. This path can be similarly constructed to the conjugate of the one shown in \cref{fig:fig1}, with all its advantages.

Apart from the sharper error bound this alternative expansion coming from (\ref{eq61a}) only differs by the presence of an exponentially small term, namely the Airy function $\mathrm{Ai} \left(u^{2/3}z\right)$. Although negligible in the Poincar\'{e} sense, its inclusion significantly increases accuracy. This is a common situation in asymptotics; see \cite[Chap. 3, Sect. 6.2]{Olver:1997:ASF}.

In summary, each expansion for $w^{(j,k)}(u,z)$ should only be used for $0\leq \arg(z) \leq 2\pi/3$  ($(j,k)=(0,1)$), $-2\pi/3\leq \arg(z) \leq 0$  ($(j,k)=(-1,0)$) and $|\arg(-z)| \leq \pi/3$  ($(j,k)=(-1,1)$), with appropriate connection formulas being employed outside these sectors. A similar argument can be made in the more general case of \cref{gz}.

Before giving an asymptotic expansion near the turning point, we first give an exact expression for the solutions, including the connection coefficient we used above. This reads as follows.
\begin{proposition}
\label{thmpoly}
Solutions of (\ref{eq70}) with $p(z)$ given by (\ref{eq75}) are explicitly given by
\begin{equation}
\label{eq98m}
w^{(j,k)}(u,z)=\gamma(u) \mathrm{Wi}^{(j,k)}
\left(u^{2/3}z \right)
+\mathcal{G}(u,z),
\end{equation}
where
\begin{equation}
\label{eq101m}
\gamma(u)=
\frac{1}{u^{4/3}}\sum\limits_{s=0}^{\left\lfloor {R/3} \right\rfloor } 
\frac{(3s)!p_{3s} }{3^{s}s!u^{2s}},
\end{equation}
and
\begin{equation}
\label{eq102m}
\mathcal{G}(u,z)=\sum\limits_{r=0}^R {u^{-2(r+2)/3}p_{r} Q_{r-1} \left(u^{2/3}z \right)}.
\end{equation}
\end{proposition}

\begin{proof}
Firstly, as noted above $z=\zeta$, $\mathcal{A}_{2m+2}(u,z)=1$ and $\mathcal{B}_{2m+2}(u,z)=0$. Thus the form (\ref{eq98m}) follows from \cref{ThmG}. To establish (\ref{eq101m}) and (\ref{eq102m}) we use the superposition principle
\begin{equation}
\label{eq97}
w^{(j,k)}(u,z)=\sum\limits_{r=0}^R {p_{r} w_{r}^{(j,k)}(u,z)},
\end{equation}
where $w_{r}^{(j,k)}(u,z)$ are solutions of
\begin{equation}
\label{eq76}
\frac{d^{2}w_{r}}{dz^{2}}-u^{2}zw_{r}=z^{r}.
\end{equation}
If we replace $z$ by $u^{-2/3}z$ we find that $w_{r}(u,z)=u^{-2(r+2)/3}y_{r}\left(u^{2/3}z \right)$, where $y_{r}(z)$ satisfies the differential equation
\begin{equation}
\label{eq77}
\frac{d^{2}y_{r}}{dz^{2}}-zy_{r}=z^{r}.
\end{equation}
Thus on using (\ref{eq72}) with $u=1$ and $p(t)=t^r$ one particular solution of (\ref{eq76}) is given by
\begin{equation}
\label{eq87}
w_{r}^{(-1,1)}(u,z)=u^{-2(r+2)/3}y_{r}^{(-1,1)} \left(u^{2/3}z \right),
\end{equation}
where
\begin{multline}
\label{eq88}
y_{r}^{(-1,1)}(z)=2\pi i
\left[ 
\mathrm{Ai}_{-1}(z)\int_{\infty \exp (2\pi 
i/3)}^{z} {t^{r}\mathrm{Ai}_{1}(t)dt} 
\right.
\\ \left.
-\mathrm{Ai}_{1}(z)\int_{\infty \exp (-2\pi i/3)}^{z} {t^{r}\mathrm{Ai}_{-1}(t)dt} 
\right],
\end{multline}
with $y_{r}^{(0,1)}(z)$ and $y_{r}^{(-1,0)}(z)$ being similarly defined.

Next, on referring to (\ref{eq4}), (\ref{eq22a}) and (\ref{eq81}), we obtain
\begin{equation}
\label{eq89}
y_{r}^{(-1,1)}(z)=Q_{r-1}(z)+c_{r} \mathrm{Wi}^{(-1,1)}(z).
\end{equation}

Likewise one finds similar expressions for $y_{r}^{(0,1)}(z)$ and $y_{r}^{(-1,0)}(z)$
and hence solutions of (\ref{eq76}) are given by
\begin{equation}
\label{eq93}
w_{r}^{(j,k)}(u,z)=u^{-2(r+2)/3}\left[ {Q_{r-1} \left(u^{2/3}z
\right)+c_{r} \mathrm{Wi}^{(j,k)}\left(u^{2/3}z \right)} \right].
\end{equation}
Now this and (\ref{eq26}) implies
\begin{equation}
\label{eq99}
w_{r}^{(0,1)}(u,z)=w_{r}^{(-1,0)}(u,z)+2\pi i c_{r} 
u^{-2(r+2)/3}\mathrm{Ai}_{1} \left(u^{2/3}z \right),
\end{equation}
and hence from (\ref{eq97}) we get (\ref{eq61a}) where
\begin{equation}
\label{eq101}
\gamma(u)=\sum\limits_{r=0}^R \frac{c_{r} p_{r} }{u^{2(r+2)/3}}.
\end{equation}
Then from (\ref{eqcr}) we arrive at (\ref{eq101m}), as asserted.

Finally from (\ref{eq97}), (\ref{eq93}) and (\ref{eq101}) we have
\begin{equation}
\label{eq98}
w^{(j,k)}(u,z)=\gamma(u) \mathrm{Wi}^{(j,k)}\left( 
u^{2/3}z \right)
+\sum\limits_{r=0}^R u^{-2(r+2)/3}p_{r} Q_{r-1} \left(u^{2/3}z \right),
\end{equation}
and on comparing this with (\ref{eq98m}) we obtain (\ref{eq102m}).
\end{proof}

Asymptotic approximations useful for $|z|\leq r_{0}$ for fixed $r_{0}>0$ are furnished by the following.

\begin{theorem}
Solutions of (\ref{eq70}) with $p(z)$ given by (\ref{eq75}) are given by (\ref{eq98m}) and (\ref{eq101m}) where
\begin{equation}
\label{eq115p}
\mathcal{G}(u,z)=
\frac{1}{u^{2}}\sum\limits_{s=0}^{m-1} \frac{\hat{G}_{s}^{\ast}(z)}{u^{2s}} +\mathcal{O}\left( \frac{1}{u^{2m+2}} \right),
\end{equation}
in which
\begin{equation}
\label{eq116p}
\hat{G}_{0}^{\ast}(z)=-\sum\limits_{r=1}^R p_{r} 
z^{r-1},
\end{equation}
\begin{equation}
\label{eq117p}
\hat{G}_{1}^{\ast}(z)
=-\sum\limits_{r=4}^R (r-1)(r-2)p_{r} z^{r-4} \quad (R\geq4),
\end{equation}
with subsequent terms being the analytic parts of the coefficients given by (\ref{eq96}).
\end{theorem}
\begin{remark}
$\hat{G}_{s}^{\ast}(z)=0$ for $3s\geq R$. For example, $\hat{G}_{2}^{\ast}(z)=\hat{G}_{3}^{\ast}(z)=\cdots=0$ for $R\leq 6$.
\end{remark}

\begin{proof}
From (\ref{eq65}) (with each $\mathcal{E}_{s}(z)$ and $\tilde{\mathcal{E}}_{s}(z)$ identically zero) and (\ref{eq66}) we have
\begin{equation}
\label{eq114p}
\mathcal{G}(u,z)=\frac{1}{u^{2}}\sum\limits_{s=0}^{m-1} 
\frac{\hat{G}_{s}(z)}{u^{2s}} +\frac{\gamma 
(u)}{u^{2/3}z}\sum\limits_{s=0}^{m-1} \frac{(3s)!}{s!\left( {3u^{2}z^{3}} 
\right)^{s}} +\mathcal{O}\left(\frac{1}{u^{2m+2}} \right).
\end{equation}
So near turning point from (\ref{eq67}) and (\ref{eq68a}) we get (\ref{eq115p}), noting that the analytic part of the second term on the RHS of (\ref{eq114p}) is identically zero. The analytic part of (\ref{eq95}) gives (\ref{eq116p}), and from (\ref{eq96}) with $s=0$
\begin{equation}
\hat{G}_{1}(z)
=-\sum\limits_{r=0}^R (r-1)(r-2)p_{r} z^{r-4},
\end{equation}
whose analytic part is given by (\ref{eq117p}).
\end{proof}

\subsection{Exponential forcing term}
\label{secexp}

Now consider $p(z)$ an exponential function, so that
\begin{equation}
\label{eq103}
\frac{d^{2}w}{dz^{2}}-u^{2}zw=e^{\alpha z},
\end{equation}
for some fixed $\alpha \in \mathbb{C}$. Recalling $z=\zeta$ we have $\xi=(2z/3)^{3/2}$. For points bounded away from the turning point $z=0$ our asymptotic approximation reads as follows.
\begin{theorem}
\label{thmexp}
Let $z \neq 0$ be a point which lies on an $R_{2}$ path $\hat{\mathcal{L}}_{\alpha}^{(j,k)}(z)$ connecting $z^{(j)}$ with $z^{(k)}$ having the property that the real part of $u(2t/3)^{3/2}+\alpha t$ is monotonic as $t$ passes along this path from $z^{(j)}$ with $z^{(k)}$. Then, for the set of all such points $z$, solutions of (\ref{eq103}) are given by (\ref{eq42}) where
\begin{equation}
\label{eq104}
\hat{G}_{0}(z)=-z^{-1}e^{\alpha z},
\end{equation}
and subsequent terms given by (\ref{eq96}). The associated error terms are $\mathcal{O}(u^{-2n-2})$ as $u \rightarrow \infty$, and are bounded by
\begin{multline}
\left\vert \hat{\varepsilon}_{n}^{(j,k)}(u,z)\right\vert \\ \leq \frac{1}{u^{2n+2}}\left\{\left\vert h(u,z)\hat{G}_{n}(z)\right\vert +
\frac{\left\vert e^{\alpha z}\right\vert }{2|z|^{1/4}}\int_{\hat{\mathcal{L}}
^{(j,k)}(z)}{\left\vert {\left\{ t^{1/4}e^{-\alpha t}h(u,t){\hat{G}}_{n}(t)\right\} ^{\prime }dt}\right\vert }\right\} \\
+\dfrac{5\left\vert e^{\alpha z}\right\vert \hat{L}_{\alpha,n}^{(j,k)}(u,z)}{32u^{2n+3}|z|^{1/4}}\left\{ 1-\dfrac{5}{32u}\int_{\hat{\mathcal{L}}_{\alpha}
^{(j,k)}(z)}{\left\vert
\frac{h(u,t)}{t^{5/2}} dt\right\vert }\right\}
^{-1}
\int_{\hat{\mathcal{L}}_{\alpha}^{(j,k)}(z)}\left\vert
\frac{h(u,t)}{t^{5/2}} dt
\right\vert,
\label{eq104p}
\end{multline}
where 
\begin{equation}
h(u,z)=\left( 1+\frac{\alpha }{2uz^{1/2}}\right)^{-1}, \label{eqh}
\end{equation}
and 
\begin{multline}
\hat{L}_{\alpha,n}^{(j,k)}(u,z)=\sup_{t\in \hat{\mathcal{L}}^{(j,k)}(z)}\left
\vert t^{1/4}e^{-\alpha t}h(u,t){\hat{G}}_{n}(t)\right\vert \\
+\frac{1}{2}\int_{\hat{\mathcal{L}}_{\alpha}^{(j,k)}(z)}{\left\vert {\left\{
t^{1/4}e^{-\alpha t}h(u,t){\hat{G}}_{n}(t)\right\} ^{\prime }dt}\right\vert }.
\label{eq105p}
\end{multline}
In this, $u$ must be sufficiently large so that for $t \in \hat{\mathcal{L}}^{(j,k)}(z)$ we have $\left\vert  \alpha  u^{-1} t^{-1/2} \right\vert <2$ (hence $h(u,t)$ is finite), and also
\begin{equation}
\label{eq106p}
\dfrac{5}{32u}\int_{\hat{\mathcal{L}}_{\alpha}
^{(j,k)}(z)}{\left\vert
\frac{h(u,t)}{t^{5/2}} dt\right\vert }<1.
\end{equation}

Asymptotic expansions valid for $|z|\leq r_{0}$ ($r_{0}>0$) are given by (\ref{eq98m}) where
\begin{equation}
\label{eq113}
\gamma(u)=u^{-4/3}\exp \left( \tfrac{1}{3}u^{-2}\alpha^{3}\right),
\end{equation}
and $\mathcal{G}(u,z)$ again has the expansion (\ref{eq115p}), where this time the analytic parts of the first two coefficients are
\begin{equation}
\label{eq116}
\hat{G}_{0}^{\ast}(z)=\frac{1}{z}-\frac{e^{\alpha z}}{z}=-\alpha 
-\frac{\alpha^{2}z}{2}-\frac{\alpha^{3}z^{2}}{6}-\cdots,
\end{equation}
\begin{equation}
\label{eq117}
\hat{G}_{1}^{\ast}(z)=\frac{2}{z^{4}}+\frac{\alpha 
^{3}}{3z}-\frac{e^{\alpha z}}{z^{4}}\left( {\alpha^{2}z^{2}-2\alpha z+2} 
\right)=-\frac{\alpha^{4}}{4}-\frac{\alpha^{5}z}{10}-\frac{\alpha 
^{6}z^{2}}{36}-\cdots.
\end{equation}
\end{theorem}

\begin{proof}
The forms (\ref{eq42}) and (\ref{eq98m}) follow immediately from the previous subsection since $f(z)$ and $g(z)$ are the same. The coefficients $\hat{G}_{s}(z)$ and $\hat{G}_{s}^{\ast}(z)$ in the theorem follow in a similar manner using $p(z)=e^{\alpha z}$. 

The bound for (\ref{eq104p}) can be derived similarly to (\ref{LG83}), with $e^{\alpha z}$ playing the role of $e^{c\xi}$, and noting that $\xi=\frac{2}{3}z^{3/2}$, $\varepsilon _{n}^{(j,k)}(u,\xi)=z^{-1/4}\hat{\varepsilon}_{n}^{(j,k)}(u,z)$ and $G_{s}(\xi)=z^{-1/4}\hat{G}_{s}(z)$. We omit details, but remark that the step involving integration by parts is aided by introducing the function defined by (\ref{eqh}) and then using (to within an arbitrary integration constant)
\begin{equation}
\int \frac{e^{u\xi +\alpha z}}{h(u,z)}d\xi =\frac{e^{u\xi +\alpha z}}{u},
\end{equation}
in which $z=(\tfrac{3}{2}\xi)^{2/3}$.

It remains to prove (\ref{eq113}). To this end, assume temporarily that $\Re (\alpha) >0$ and set $z=0$ and $p(t)=e^{\alpha t}$ in (\ref{eq72}) and (\ref{eq74}). If we deform the contours with end points $\infty \exp (\pm 2\pi i/3)$ so that they lie on the negative real axis we then get
\begin{multline}
\label{eq107}
 w^{(-1,1)}(u,0)=2\pi iu^{-2/3} \mathrm{Ai}(0)
 \left[ 
 \int_{-\infty }^{0} {e^{\alpha 
t}\mathrm{Ai}_{1} \left(u^{2/3}t \right)dt} 
\right.
\\ \left.
 -\int_{-\infty }^{0} {e^{\alpha t}\mathrm{Ai}_{-1} 
\left(u^{2/3}t \right)dt} \right],
\end{multline}
and
\begin{multline}
\label{eq108}
 w^{(0,1)}(u,0)=2\pi e^{-\pi i/6}u^{-2/3}\mathrm{Ai}(0)
 \left[
 \int_{\infty}^{0}
{e^{\alpha t}\mathrm{Ai}\left(u^{2/3}t \right)dt}
\right.
\\ 
\left.
 -\int_{-\infty }^{0} {e^{\alpha t}\mathrm{Ai}_{1} \left(u^{2/3}t \right)dt}
 \right].
\end{multline}
As a result we have 
\begin{multline}
\label{eq110}
 w^{(-1,1)}(u,0)-w^{(0,1)}(u,0) \\ =2\pi u^{-2/3}\mathrm{Ai}(0) \left[ \int_{-\infty }^{0} {e^{\alpha t}\left\{\left( i+e^{-\pi
i/6} \right)\mathrm{Ai}_{1} \left(u^{2/3}t \right)-i\mathrm{Ai}_{-1} \left(u^{2/3}t \right) \right\}dt} \right. \\ \left. +e^{-\pi i/6}\int_{0}^{\infty} {e^{\alpha t}\mathrm{Ai}\left(u^{2/3}t \right)dt} \right].
\end{multline}
Hence from (\ref{eq25})
\begin{equation}
\label{eq111}
w^{(-1,1)}(u,0)-w^{(0,1)}(u,0)=2\pi e^{-\pi 
i/6}u^{-2/3}\mathrm{Ai}(0)\int_{-\infty }^{\infty} {e^{\alpha t}\mathrm{Ai}\left( 
u^{2/3}t \right)dt}.
\end{equation}
Now from \cite[Eq. 9.10.13]{NIST:DLMF} we find the exact expression
\begin{equation}
\label{eq112}
\int_{-\infty }^{\infty} {e^{\alpha t}\mathrm{Ai}\left(u^{2/3}t \right)dt} 
=\frac{1}{u^{2/3}}\exp \left(\frac{\alpha^{3}}{3u^{2}} 
\right).
\end{equation}
But from setting $z=0$ in (\ref{eq54}) (and noting $m$ dependence in the present situation) we have
\begin{equation}
\label{eq109}
w^{(-1,1)}(u,0)-w^{(0,1)}(u,0)=2\pi e^{-\pi i/6}\gamma(u)\mathrm{Ai}(0).
\end{equation}
Hence from (\ref{eq111}) - (\ref{eq109}) we arrive at (\ref{eq113}). Finally, by analytic continuation we can relax the restriction on $\alpha$.
\end{proof}

\subsection{Integrals involving the Airy function and the exponential function}
\label{airyint}

Another application of our asymptotic approximations for particular solutions of (\ref{eq1}) is extracting this for the integrals that appear in their variation of parameters representations.

Let us illustrate this first with the solution $w^{(-1,0)}(u,z)$ of (\ref{eq69}), recalling that we assume $u>0$. Using (\ref{eq73}) and its differentiated form
\begin{multline}
\label{eq118}
\frac{\partial w^{(-1,0)}(u,z)}{\partial z}=2\pi e^{\pi i/6}
\left[\mathrm{Ai}'_{-1} 
\left(u^{2/3}z \right)\int_{\infty}^{z} {p(t)\mathrm{Ai}\left(u^{2/3}t 
\right)dt} 
\right.
\\ 
\left.
 -\mathrm{Ai}'\left(u^{2/3}z \right)\int_{\infty \exp
(-2\pi i/3)}^{z} {p(t)\mathrm{Ai}_{-1} \left(u^{2/3}t \right)dt}
\right],
\end{multline}
we find by solving for the integral involving $\mathrm{Ai}(u^{2/3}t)$, and using (\ref{eq71}), that
\begin{multline}
\label{eq119}
\int_{\infty}^{z} {p(t)\mathrm{Ai}\left(u^{2/3}t \right)dt} =\mathscr{W}\left\{ 
{\mathrm{Ai}\left(u^{2/3}z \right),w^{(-1,0)}(u,z)} \right\} \\ 
= \mathrm{Ai}\left(u^{2/3}z \right)
\frac{\partial w^{(-1,0)}(u,z)}{\partial z}
-u^{2/3}{\mathrm{Ai}}'\left(u^{2/3}z 
\right)w^{(-1,0)}(u,z).
\end{multline}

One similarly obtains
\begin{equation}
\label{eq120}
\int_{\infty}^{z} {p(t)\mathrm{Ai}\left(u^{2/3}t \right)dt} =\mathscr{W}\left\{ 
{\mathrm{Ai}\left(u^{2/3}z \right),w^{(0,1)}(u,z)} \right\},
\end{equation}
and from (\ref{eq71}), (\ref{eq119}), along with the connection formula
\begin{equation}
\label{eq120a}
w^{(-1,0)}(u,z)=
w^{(-1,1)}(u,z)-2\pi e^{\pi i/6}\gamma(u)
\mathrm{Ai}_{-1} \left(u^{2/3}z\right),
\end{equation}
the third representation
\begin{equation}
\label{eq121}
\int_{\infty}^{z} {p(t)\mathrm{Ai}\left(u^{2/3}t \right)dt} =\mathscr{W}\left\{ 
\mathrm{Ai}\left(u^{2/3}z \right),w^{(-1,1)}(u,z) \right\}
-u^{2/3}\gamma(u).
\end{equation}

The method is to now substitute our new uniform asymptotic expansions into the appropriate one of these three exact expressions, depending on which part of the complex plane $z$ lies in, and likewise replace the Airy function (and its derivative) by its asymptotic expansion if $z$ is bounded away from the origin. 

We shall do this for $p(t)=e^{\alpha t}$. Then for $|z| \ge r_{0}$ ($r_{0}>0$) and $\vert \arg(z)\vert \le \pi -\delta$ we use \cref{thmexp} along with (\ref{eq119}), (\ref{eq120}) and \cite[Sect. 9.7]{NIST:DLMF}. Consequently, uniformly for $u \rightarrow \infty$, we derive
\begin{multline}
\label{eq122}
 \int_z^{\infty} {e^{\alpha t}\mathrm{Ai}\left(u^{2/3}t \right)dt} 
=\frac{1}{2\sqrt \pi u^{7/6}z^{3/4}}\exp \left( {\alpha z-\frac{2}{3}uz^{3/2}} \right) \\ 
 \times \left\{ {1+\frac{48\,\alpha z-41}{48uz^{3/2}}+\frac{4608\,\alpha 
^{2}z^{2}-9696\,\alpha z+9241}{4608u^{2}z^{3}}+\frac{S(\alpha 
z)}{u^{3}z^{9/2}}+\mathcal{O}\left(\frac{1}{u^{4}z^{2}} \right)} \right\},
\end{multline}
where
\begin{equation}
\label{eq123}
S(t)=48t^{3}-185t^{2}+\frac{35905}{96}t-\frac{5075225}{13824}.
\end{equation}

If we use (\ref{eq121}) instead of (\ref{eq119}) and (\ref{eq120}), then for $|z| \ge r_{0}$ and $\vert \arg(z)\vert \le \tfrac{2}{3}\pi -\delta $ we obtain uniformly for $u \rightarrow \infty$
\begin{multline}
\label{eq124}
 \int_{-z}^{\infty} {e^{\alpha t}\mathrm{Ai}\left(u^{2/3}t \right)dt}
=\frac{1}{u^{2/3}}\exp \left(\frac{\alpha^{3}}{3u^{2}} 
\right) \\+\frac{e^{-\alpha z}}{\sqrt \pi u^{7/6}z^{3/4}}  \left[ \sin \left( {\frac{2}{3}uz^{3/2}-\frac{1}{4}\pi} \right)\left\{
{1-\frac{4608\,\alpha^{2}z^{2}+9696\,\alpha z+9241}{4608u^{2}z^{3}}+\mathcal{O}\left( 
\frac{1}{u^{4}z^{2}} \right)} \right\} \right. \\ 
-\frac{1}{48uz^{3/2}}\left. \cos \left( {\frac{2}{3}uz^{3/2}-\frac{1}{4}\pi} 
\right)\left\{ 48\alpha z+41+\frac{S(-\alpha z)}{u^{2}z^{3}}+\mathcal{O}\left( 
\frac{1}{u^{4}z} \right) \right\} \right].
\end{multline}

Finally for $z\leq r_{0}$ we do not expand the Airy functions, and again from (\ref{eq119}), (\ref{eq120}), (\ref{eq121}) and \cref{thmexp} we have uniformly as $u \rightarrow \infty$
\begin{multline}
\label{eq125}
\int_z^{\infty} e^{\alpha t} \mathrm{Ai}\left(u^{2/3}t \right) dt 
=\frac{1}{u^{2/3}}\exp \left(\frac{\alpha^{3}}{3u^{2}}
\right) \\ \times
\left[ {\mathrm{Ai}}' \left(u^{2/3}z \right) \mathrm{Wi}^{(j,k)} \left(u^{2/3}z \right) - \mathrm{Ai} \left(u^{2/3}z \right) {\mathrm{Wi}^{(j,k)}}' \left(u^{2/3}z \right)
+\delta_{-j,k}      \right] \\
+\frac{1}{u^{4/3}}{\mathrm{Ai}}'\left(u^{2/3}z \right)
\left\{\sum\limits_{s=0}^{m-1} \frac{\hat{G}_{s}^{\ast}(z)}{u^{2s}}+\mathcal{O}\left( \frac{1}{u^{2m}}\right)
\right\}
\\-\frac{1}{u^{2}}\mathrm{Ai}\left(u^{2/3}z \right)
\left\{\sum\limits_{s=0}^{m-1} 
\frac{{\hat{G}_{s}^{\ast \prime}}(z)}{u^{2s}}
+\mathcal{O}\left( \frac{1}{u^{2m}} \right)
\right\},
\end{multline}
where $\delta$ is the Kronecker delta function, and the first two coefficients in the series are given by (\ref{eq116}) and (\ref{eq117}). For optimal accuracy one should take $(j,k)=(0,1)$ for $0\leq \arg(z) \leq 2\pi/3$, $(j,k)=(-1,0)$ for $-2\pi/3\leq \arg(z) \leq 0$, and $(j,k)=(-1,1)$ for $|\arg(-z)| \leq \pi/3$.

\section*{Acknowledgments} I thank the two anonymous referees for a number of helpful comments that improved the manuscript. Financial support from Ministerio de Ciencia e Innovaci\'on, Spain, project PGC2018-098279-B-I00 (MCIU/AEI/FEDER, UE) is acknowledged. 

\bibliographystyle{siamplain}
\bibliography{biblio}

\end{document}

%% file: ex_shared.tex
\usepackage{lipsum}
\usepackage{mathrsfs}
\usepackage{amsfonts}
\usepackage{graphicx}
\usepackage{epstopdf}
\usepackage{algorithmic}
\usepackage{amssymb}

\ifpdf
  \DeclareGraphicsExtensions{.eps,.pdf,.png,.jpg}
\else
  \DeclareGraphicsExtensions{.eps}
\fi

\newsiamremark{hypothesis}{Hypothesis}
\crefname{hypothesis}{Hypothesis}{Hypotheses}
\newsiamthm{claim}{Claim}
\newtheorem{remark}{Remark}
\nolinenumbers

\headers{Inhomogeneous ODEs having a turning point}{T. M. Dunster}

\title{Asymptotic solutions of inhomogeneous differential equations having a turning point}

\author{T. M. Dunster\thanks{Department of Mathematics and Statistics, San Diego State University, 5500 Campanile Drive, San Diego, CA 92182, USA. 
  (\email{mdunster@sdsu.edu}, \url{https://tmdunster.sdsu.edu}).}}

\usepackage{amsopn}

\makeatletter
\newcommand*{\addFileDependency}[1]{
  \typeout{(#1)}
  \@addtofilelist{#1}
  \IfFileExists{#1}{}{\typeout{No file #1.}}
}
\makeatother
